\documentclass[dvipdfmx,12pt,a4paper,reqno]{amsart}
\usepackage{amssymb}
\usepackage{amsmath}
\usepackage[mathscr]{euscript}
\usepackage{times}
\usepackage{newcent}
\usepackage{graphicx}
\usepackage{hyperref}

\usepackage{color}
\usepackage[top=30truemm,bottom=30truemm,left=25truemm,right=25truemm]{geometry}
\usepackage{tikz}
\usepackage{comment}

\allowdisplaybreaks[4]

\makeatletter
\@namedef{subjclassname@2020}{\textup{2020} Mathematics Subject Classification}
\makeatother
\makeatletter
\@addtoreset{equation}{section}
\makeatother

\newtheorem{theorem}{Theorem}[section]
\newtheorem{lemma}[theorem]{Lemma}
\newtheorem{proposition}[theorem]{Proposition}

\newtheorem{remark}[theorem]{Remark}

\newcommand{\mc}[1]{{\mathcal #1}}

\newcommand{\bb}[1]{{\mathbb #1}}

\newcommand{\e}{\varepsilon}

\newcommand{\lan}{\langle}
\newcommand{\ran}{\rangle}

\newcommand{\N}{\bb N}
\newcommand{\Z}{\bb Z}

\newcommand{\R}{\bb R}

\newcommand{\T}{\bb T}

\newcommand{\ac}{{\rm ac}}

\renewcommand{\tilde}{\widetilde}
\renewcommand{\hat}{\widehat}

\renewcommand{\u}{\bar u}
\renewcommand{\v}{\bar v}
\renewcommand{\L}{L_{\u}}

\begin{document}

\title[Sharp interface limit for a quasilinear functional]{Sharp interface limit for a quasi-linear large deviation rate function}

\author[T. Kagaya]{Takashi Kagaya}
\address[T. Kagaya]{Graduate school of Engineering, Muroran Institute of Technology, 27-1 Mizumoto-cho, Hokkaido, 050-8585, Japan}
\email{kagaya@muroran-it.ac.jp}

\author[K. Tsunoda]{Kenkichi Tsunoda}
\address[K. Tsunoda]{Faculty of Mathematics, Kyushu University, 744 Motooka, Nishi-ku, Fukuoka, 819-0395, Japan}
\email{tsunoda@math.kyushu-u.ac.jp}

\subjclass[2020]{Primary 34D15; 
Secondly, 60F10, 
82C22.
}
\keywords{Sharp interface limit, reaction-diffusion equation, motion by mean curvature, large deviations.}
\date{\today}

\maketitle

\begin{abstract}
We discuss the sharp interface limit, leading to a mean curvature flow energy, for the rate function of the large deviation principle of a Glauber+Kawasaki process with speed change. 
We provide an explicit formula of the limiting functional given by the mobility and the transport coefficient.
\end{abstract}

\section{Introduction}\label{sec1}
The investigation of a macroscopic evolution such as heat flow, fluid motion, etc. from a microscopic system is a central problem in statistical physics. In the past decades, the large deviation principle for a stochastic system and its rate function have attracted attention not only in probability theory but also in mathematical physics. For instance, the so-called macroscopic fluctuation theory, which is based on a description of large deviations for a microscopic system, has been successfully applied to a study of stationary nonequilibrium states, cf. \cite{BDGJL}. At the same time, a study of large deviations has provided rich mathematical problems for partial differential equations and variational problems.

Our main concern is the sharp interface limit for a Glauber+Kawasaki process with speed change. For this purpose, we start by defining the late function of the large deviation principle for the process. Let $\T_N^d$ be the $d$-dimensional discrete torus with length $N$, that is, $\T_N^d=(\Z/N\Z)^d$ (a sum in $\T_N^d$ is taken modulo $N$). Here, $N\in\N$ is a scaling parameter which we will let infinity later. Let $X_N$ be the configuration space $\{0,1\}^{\T_N^d}$ and denote its generic element as $\eta=\{\eta(x)\}_{\T_N^d}$. We regard a configuration $\eta\in X_N$ in the following manner: for each site $x\in\T_N^d$, there is a particle at $x$ if $\eta(x)=1$, otherwise, there is no particle at site $x$.

We now define the Markovian generator $L_N$ defined as $L_Nf=N^2L_Kf+ \tilde{K}L_Gf$ for any function $f:X_N\to\R$, where $L_K$ and $L_G$ are given by
\begin{align*}
L_Kf(\eta)&=\sum_{x\in\T_N^d}\sum_{i=1}^dc_{0,e_i}(\tau_x\eta)\{f(\eta^{x,y}) - f(\eta)\},\\
L_Gf(\eta)&=\sum_{x\in\T_N^d}c_0(\tau_x\eta)\{f(\eta^x) - f(\eta)\},
\end{align*}
and $\tilde{K}$ is a positive constant.
In the previous formulas, $\tau_x\eta, \eta^{x,y}$ and $\eta^x$ are the configurations obtained from a configuration $\eta$ by translating by $x$, by exchanging the occupation variables at $x$ and $y$ and by flipping the occupation variable at $x$, respectively: $\tau_x\eta(z)=\eta(x+z)$,
\begin{align*}
\eta^{x,y}(z)=
\begin{cases}
\eta(y), \quad \text{if $z=x$},\\
\eta(x), \quad \text{if $z=y$},\\
\eta(z), \quad \text{otherwise},
\end{cases}
\quad
\eta^x(z)=
\begin{cases}
1-\eta(x), \quad \text{if $z=x$},\\
\eta(z), \quad \text{otherwise}.
\end{cases}
\end{align*}
Let $(\eta_t^N)_{t\ge0}$ denote a Markov process generated by $L_N$. Although one needs to assume several conditions for $c_{0,e_j}$ and $c_0$ for the validity of the following results, we do not go into the details here. See \cite{FvMST1} for the details. Let $\T^d$ be the $d$-dimensional continuum torus $(\R/\Z)^d$. We define the empirical measure by
\begin{align*}
\pi_t^N(du)=\dfrac{1}{N^d}\sum_{x\in\T_N^d}\eta_t^N(x)\delta_{x/N}(du),
\end{align*}
where $\delta_u, u\in\T^d$ stands for the Dirac measure at $u\in\T^d$.

The scaling limit for empirical measures is a fundamental problem in the study of interacting particle systems. Under an appropriate convergence on the initial distributions and the so-called gradient condition for $c_{0,e_j}$, it is known from \cite{DFL,FHU} that the empirical measure $\pi_t^N(dx)$ converges to the deterministic measure $\rho(t,x)dx$ as $N\to\infty$, where $\rho(t,x)$ is the unique weak solution to the reaction-diffusion equation
\begin{align}\label{intro6}
\partial_t\rho=\Delta P(\rho) + \tilde{K}(B(\rho) - D(\rho)).
\end{align}
Here, $P, B, D$ are some functions which are determined from $c_{0,e_j}$ and $c_0$.
When the jump rate $c_0$ of the Glauber dynamics corresponds to the Ising model in a low-temperature regime,
the reaction term $B-D$ becomes a derivative of a double-well potential having exactly two local minima,
which we will denote $\rho_\pm$. See \cite{DFL} for the details.

For the Glauber+Kawasaki process, a large deviation principle, which determines the decay rate for the probability of an atypical event of the system, has also been studied in \cite{JLV, BFG, LT}. Loosely speaking, for a given density evolution $\phi:[0,T]\times\T^d\to[0,1]$, the probability that the empirical measure $\pi_t^N(dx)$ follows $\phi(t,x)dx$ behaves as
\begin{align*}
\mathbb P\left(\pi^N_\cdot\sim \phi(\cdot,x)dx\right) \approx \exp \{-N^d S(\phi)\},
\end{align*}
where $S(\phi)$ is given by
 \begin{align*}
S(\phi) &=\sup_{H\in C^{1,2}([0,T]\times \T^d)} J^H(\phi),\\
J^H(\phi) 
&= \int_{\T^d} \phi(T,x)H(T,x) \; dx - \int_{\T^d} \phi(0,x)H(0,x) \; dx\\
& \quad - \int_0^T \int_{\T^d} \left\{ \phi\partial_t H + P(\phi)\Delta H + \sigma(\phi)|\nabla H|^2 \right\} \; dxdt\\
& \qquad - \int_0^T \int_{\T^d} \tilde{K}\left\{ B(\phi)\left(e^{H}-1\right) +  D(\phi)\left(e^{-H}-1\right) \right\} \; dxdt.
\end{align*}
Note that $S(\phi)$ vanishes if and only if $\phi$ solves the reaction-diffusion equation \eqref{intro6}.

Letting $\e := 1/ \sqrt{\tilde{K}}$, the reaction-diffusion equation \eqref{intro6} introduces an Allen-Chan type equation 
\begin{equation}\label{intro7} 
\partial_t \rho_\e = \Delta P(\rho_\e) + \frac{1}{\e^2} (B(\rho_\e) - D(\rho_\e)).
\end{equation}
Heuristically, at each time, $\rho_\e$ is close to a step function for sufficient small $\e$, and the transition layer converges to a surface $\Gamma_t$ generating a mean curvature flow with a mobility constant $\theta$ determined by $P, B$ and $D$ as $\e \to +0$, namely, the motion of $\Gamma_t$ is governed by $v_t - \theta h_t$, where $v_t$ and $h_t$ are the normal velocity and the mean curvature of $\Gamma_t$, respectively. 
In particular, the transition layer can be represented as $\rho_\varepsilon(t, x) \approx \overline{u}(d(t,x)/ \e)$, where $d(t,x)$ is a signed distance function from $\Gamma_t$ and $\overline{u}$ is a solution to the ordinary differential equation 
\begin{equation}\label{pro-translayer}
\begin{cases}
\left(P(\bar u)\right)'' +B(\bar u) - D(\bar u)=0 \quad \text{in} \; \; \R,\\
\bar u(\pm\infty)=\rho_\pm, \quad \bar u(0)=\dfrac{\rho_++\rho_-}{2}.
\end{cases}
\end{equation}
For the known convergence results, we refer to \cite{Ch, ESS, BSS, Il, So} for the case $P(\rho) = \rho$ and \cite{EFHPS} for more general $P(\rho)$.

For the case $P(\rho) = \rho/2$, Bertini, Butt\`{a} and Pisante \cite{BBP} characterized the functional $S_\e(\phi)$ from the perspective of the sharp interface limit by substituting a family of functions generating a transition layer around an arbitrary fixed geometric flow into the functional.
Our purpose is to extend the result to a more general function $P$.
For this purpose, for each $\e>0$ let us define 
 \begin{align}
&S_\e(\phi) =\sup_{H\in C^{1,2}([0,T]\times \T^d)} J_\e^H(\phi), \label{def-S} \\
& \begin{aligned}
J_{\e}^H(\phi) 
&= \dfrac{1}{\e}\int_{\T^d} \phi(T,x)H(T,x) \; dx - \dfrac1\e\int_{\T^d} \phi(0,x)H(0,x) \; dx\\
& \quad - \dfrac{1}{\e}\int_0^T \int_{\T^d} \left\{ \phi\partial_t H + P(\phi)\Delta H + \sigma(\phi)|\nabla H|^2 \right\} \; dxdt\\
& \quad \quad - \dfrac{1}{\e^3}\int_0^T \int_{\T^d} \left\{ B(\phi)\left(e^{H}-1\right) +  D(\phi_t)\left(e^{-H}-1\right) \right\} \; dxdt.
\end{aligned} \label{def-J}
\end{align}
We also assume that $P: [0,1] \to [0, \infty)$ and $B,D: [0,1] \to \R$ and $W: [0,1] \to \R$ are smooth functions satisfying the following conditions: 
\begin{itemize}
\item[(A1)] $P$ satisfies $P(0)=0$ and $P'(\rho)>0$ for any $\rho\in[0,1]$.
\item[(A2)] $B(\rho)+D(\rho)$ is positive for any $\rho\in[0,1]$ and $B-D=-W'$.
\item[(A3)] $W$ is a double-well potential, that is, there exist exactly three critical points $0<\rho_-<\rho_*<\rho_+<1$ such that $W(\rho_\pm) < W(\rho)$ for any $\rho\neq\rho_\pm$ and $W''(\rho_\pm)>0$.
\item[(A4)] $W$ satisfies a {\it $P$-balance condition}, that is, it holds that 
\begin{align}\label{2-1}
\int_{\rho_-}^{\rho_+} W'(\rho)P'(\rho) \; d\rho=0. 
\end{align}
\end{itemize}
We note that the conditions (A1) and (A2) are satisfied when $B$ and $D$ are determined from a wide class of jump rates of the Glauber dynamics. Moreover, the conditions (A3) and (A4) were introduced in \cite{EFHPS} and it was clarified that these conditions are needed to obtain a sharp interface limit for \eqref{intro7} leading to the motion by mean curvature.

In these settings, our goal can be stated that, restricting the form of a family of functions $\{\phi_\e\}_{\e > 0}$ so that functions generating the transition layer around an arbitrary fixed geometric flow $\{\Gamma_t\}_{t \in [0,T]}$, we show a ``formal'' $\Gamma$-convergence from $S_\e(\phi_\e)$ to
\begin{align}\label{intro3}
S_\ac(\Gamma)=\int_0^T \int_{\Gamma_t}\dfrac{(v_t-\theta h_t)^2}{4\mu} \; d\mc H^{d-1}dt,
\end{align}
where $v_t, h_t$ are respectively the normal velocity and the mean curvature of $\Gamma_t$, $\mathcal{H}^{d-1}$ is the $(d-1)$-dimensional Hausdorff measure and $\theta, \mu$ are respectively the mobility and the transport coefficient determined by $P, B$ and $D$ (see \eqref{def-mu-theta} for details). 
To state the form of the family of functions $\{\phi_\e\}_{\e > 0}$, we define a regularized version of a signed distance function from $\Gamma_t$ as follows.
For a family of oriented smooth hyper-surfaces $\Gamma=\{\Gamma_t\}_{t\in[0,T]}$ with $\Gamma_t=\partial\Omega_t$ for some open $\Omega_t\subset \T^d$ and with the finite surface area for any $t \in [0,T]$,  
 choose $d(\cdot, t)$ as a regularized version of the signed distance from $\Gamma_t$ satisfying 
\begin{equation*}
d(t,x) = 
\begin{cases}
{\rm dist}(x, \Gamma_t) & \text{if} \; \; x \not\in \Omega_t \; \; \text{and} \; \; {\rm dist}(x, \Gamma_t) \ll 1, \\
-{\rm dist}(x, \Gamma_t) & \text{if} \; \; x \in \Omega_t \; \; \text{and} \; \; {\rm dist}(x, \Gamma_t) \ll 1. 
\end{cases}
\end{equation*}
See the end of Section \ref{sec2} for the proximity to the surface stated above. 
Then, the main result in this paper is stated as follows.

\begin{theorem}\label{thm2-1}
Assume the properties (A1)--(A4) hold. 
Let $\Gamma=\{\Gamma_t\}_{t\in[0,T]}$ be a family of oriented smooth hyper-surfaces with $\Gamma_t=\partial\Omega_t$ for some open $\Omega_t\subset \T^d$ and with the finite surface area for any $t \in [0,T]$. 
Let also $\overline{u}$ be the unique smooth solution to \eqref{pro-translayer}. 
For smooth functions $Q:[0,T]\times\T^d\times\R\to\R$ and $R_\e:[0,T]\times\T^d\to\R$, define the function $\phi_\e:[0,T]\times\T^d\to[0,1]$ by
\begin{align}\label{2-6}
\phi_\e(t,x)=\bar u\left(\dfrac{d(t,x)}{\e}+\e Q\left(t,x,\dfrac{d(t,x)}{\e}\right) \right) +\e R_\e(t,x).
\end{align}
Then we have the following.
\begin{enumerate}
\item If $Q$ and $R_\e$ satisfy
\begin{align}
&\sup_{(t, x, \xi) \in [0,T] \times \T^d \times \R} \left(\frac{|\partial_t Q|}{1+|\xi|} + \sum_{i=0}^2 \sum_{j=0}^2 \frac{|\partial_\xi^i \nabla^j Q|}{1+|\xi|} \right) < \infty, \label{as-Q}\\
&\lim_{\e \to 0} \left(\sup_{(t,x) \in [0,T] \times \T^d} (|R_\e| + |\partial_t R_\e| + |\nabla R_\e| + |\nabla^2 R_\e|) \right) = 0, \label{as-R}
\end{align}
then 
\begin{align*}
\liminf_{\e\to0} S_\e(\phi_\e) \ge S_\ac(\Gamma).
\end{align*}
\item There exists $\hat{Q}$ such that, choosing $Q=\hat{Q}$ and $R_\e=0$, it holds that 
\begin{align*}
\lim_{\e\to0}S_\e(\phi_\e)=S_\ac(\Gamma).
\end{align*}
\end{enumerate}
\end{theorem}

We note that the sequence of the functions $\phi_\e$ formed by \eqref{2-6} converges to $\rho_-$ inside $\Omega_t$ and $\rho_+$ outside $\Omega_t$, which gives us a kind of the ``convergence class'' from function to geometric flow. Further discussion of this ``convergence class'' is needed to prove a more accurate $\Gamma$-convergence theory that does not assume the regularity of the geometric flow. 
Another remark is that in the previous work \cite{BBP}, the mobility $\theta$ is specified as $1/2$, which is consistent with the sharp interface limit for \eqref{intro7}.
However, since we deal with the rate function related to the quasilinear reaction-diffusion equation \eqref{intro7}, how the mobility is determined is unclear. 
We also note that the mobility that we obtain coincides with the one appearing in the sharp interface limit for \eqref{intro7} (cf. \cite{EFHPS}). 
Finally, we mention the important difference between our analysis and the one of \cite{BBP}. 
An optimal corrector, which is needed to obtain the mobility and the transport coefficient, can not be expressed as a product of two functions, one is a function in a scaled singed distance from the interface and the other is a function in time and a point of $\T^d$. 
For this reason, we need to care about all the dependence on the latter two variables in our analysis. See also Remark \ref{rem:diff} for more details.

In the rest of this section, we examine several papers related to this work.
The sharp interface limit for the rate functions of  the large deviation principle has been originally studied in \cite{KORV}
in the case of the stochastic reaction-diffusion equation:
\begin{align}\label{intro1}
\partial_t\rho=\Delta\rho-\e^{-2}W'(\rho)+\sqrt{2\gamma}\xi_\delta,
\end{align}
where $\e, \gamma>0$ are small parameters, $W$ is a nice potential and $\xi_\delta$ is a smooth random noise with the correlation length $\delta>0$.
It is known that for fixed $\delta>0$, as $\gamma\to0$ the law of \eqref{intro1} satisfies a large deviation principle for some rate function $I_\e^{(\delta)}$. Formally, as $\delta\to0$ the rate function $I_\e^{(\delta)}$ converges to the function $I_\e$  given by
\begin{align*}
I_\e(\rho)=\dfrac14\int_0^T\int_{\T^d} \left( \partial_t\rho -\Delta\rho + \e^{-2} W'(\rho) \right)^2\; dxdt.
\end{align*}
This formal procedure has been justified in \cite{FL} in the case where the spatial dimension is one and the random noise is a space-time white noise. Higher dimensional cases are much more delicate because of fluctuations coming from the random noise. See \cite{HW} and references therein for more details.

Returning to \cite{KORV}, they showed that if a family of space-time functions $\rho_\e$ converges to a step function at each time and creates an evolution of hyper-surfaces $\Gamma=\{\Gamma_t\}_{t\in[0,T]}$ in the limit, $\e I_\e(\rho_\e)$ converges as $\e\to0$ to  $S_\ac(\Gamma)$ with the model-dependent constants $\theta$ and $\mu$. 
We also note that a $\Gamma$-limsup inequality for smooth functions has been proved in \cite{KORV} and a $\Gamma$-liminf inequality in \cite{MR}. 
In particular, in the latter result \cite{MR}, the functional $S_\ac(\Gamma)$ is defined in the sense of measure to discuss without the regularity assumption of the geometric flow.

The same problem has been considered for other models, which are studied in statistical physics. It is known that the motion by mean curvature appears in the sharp interface limit for the Glauber dynamics for Ising systems with Kac potentials and the Glauber+Kawasaki process. The large deviations for both processes are well-studied interacting systems and the sharp interface limit for the rate functions of these models has been studied in \cite{BBP}.

We finally mention several results that connect the Glauber+Kawasaki process with
the motion by mean curvature. Let us consider the case where $P(\rho)=\rho$ for a while.
As mentioned before, it is known that if the potential of the reaction term $B-D$ is
a double-well and satisfies the balance condition \eqref{2-1}, the singular limit for \eqref{intro7}
leads to the motion by mean curvature in the limit $\e\to0$, cf. \cite{BELLETTINI}.
Since the macroscopic density of the Glauber+Kawasaki process evolves
according to the reaction-diffusion equation, one may expect that, introducing
a large constant $K_N$ diverging to $\infty$ as $N\to\infty$ in front of the Glauber generator,
the motion by mean curvature appears in the scaling limit
for the Glauber+Kawasaki process. This was originally done by
the correlation function method in \cite{BONAVENTURA}
and by the level-set method in \cite{KS}.
Recently, the same procedure has been done by the entropy method in \cite{FT}.
The entropy method is more robust than the other methods
and thereby one can handle the same problem with nonlinear $P$.
Indeed, using the entropy method, the sharp interface limit has been established
in the case of the Glauber+zero-range process in \cite{EFHPS}
and the Glauber+Kawasaki process with speed change in \cite{FvMST1}, respectively.
Note that when $P(\rho)=\rho$ and the balance condition \eqref{2-1} is broken,
a different scaling must be taken and a constant-speed interface flow appears in the sharp interface limit, see \cite{FvMST2}.

This paper is organized as follows. Section \ref{sec2} lists known propositions that are necessary to prove the main theorem and several notions also defined in Section \ref{sec2}.
In Section \ref{sec3} we prove Theorem \ref{thm2-1} after preparing several auxiliary lemmas. Appendices \ref{seca} and \ref{secb} are devoted to the study of analytic and geometric properties which are necessary for the proof of Theorem \ref{thm2-1}, respectively.

\subsection*{Acknowledgments}
The first author was partially supported by Japan Society for the Promotion of Science (JSPS) through grants: KAKENHI \#20H01801, \#21H00990, \#23K12992, \#23H00085.
The second author was partially supported by Japan Society for the Promotion of Science (JSPS) through grant: KAKENHI \#22K13929.

\section{Preparations}\label{sec2}

In this section, we list known propositions that are necessary to prove the main theorem, and several notions are defined.

We first recall the existence of a unique smooth solution to the ODE \eqref{pro-translayer} and some estimates of the solution. The ODE can be re-formulated for $\bar v= P(\bar u)$ as 
\begin{align}\label{2-3}
\begin{cases}
\bar v'' +B\left(P^{-1}(\bar v)\right) - D\left(P^{-1}(\bar v)\right)=0 \quad \text{in} \; \; \R,\\
\bar v(\pm\infty)=P(\rho_\pm), \quad \bar v(0)=P\left(\dfrac{\rho_++\rho_-}{2}\right). 
\end{cases}
\end{align}
We state some properties of $\bar u$ and $\bar v$ as follows. 

\begin{lemma}\label{lem:ODE-property}
Assume that smooth functions $P: [0,1] \to [0, \infty)$, $B,D: [0,1] \to \R$ and $W: [0,1] \to \R$ satisfy the properties {\rm (A1)--(A4)}. 
Then, \eqref{pro-translayer} and \eqref{2-3} admit a unique smooth solution. 
Furthermore, the following properties hold: 
\begin{itemize}
\item[(1)] There exists $\gamma_{\max, \bar u} >0$ such that, for any $\gamma \in (0, \gamma_{\max, \bar u})$, there exists $C_{\gamma, \u} > 0$ such that 
\begin{align}
& \bar u' (\xi) > 0 \quad \text{for} \; \; \xi \in \R, \notag\\
& |\bar u'(\xi) | + |\bar u''(\xi)| + |\u'''(\xi)| \le C_{\gamma, \u} e^{-\gamma |\xi|} \quad \text{for} \; \; \xi \in \R. \label{ex-decay-u}
\end{align}
\item[(2)] Let  
\[ \gamma_{\max, \bar v} :=  \min \left\{\sqrt{\frac{W''(\rho_-)}{P'(\rho_-)}}, \sqrt{\frac{W''(\rho_+)}{P'(\rho_+)}}\right\}. \]
Then, for any $\gamma \in (0, \gamma_{\max, \bar v})$, there exists $C_{\gamma, \v} > 0$ such that 
\begin{align}
& \bar v' (\xi) > 0 \quad \text{for} \; \; \xi \in \R, \notag\\
& |\v' (\xi) | + |\v''(\xi)| + |\v'''(\xi)| \le C_{\gamma, \v} e^{-\gamma |\xi|} \quad \text{for} \; \; \xi \in \R. \label{ex-decay-v}
\end{align}
\end{itemize}
\end{lemma}

\begin{proof}
Letting $f := (D-B)\circ P^{-1}$, we can see that $f$ satisfies 
\begin{align} 
&f(P(\rho_\pm)) = 0, \quad f'(P(\rho_\pm)) >0, \notag\\
&\int_{P(\rho_-)}^{\hat{\alpha}} f(\alpha) \; d\alpha = \int_{P(\rho_+)}^{\hat{\alpha}} f(\alpha) \; d\alpha > 0 \quad \text{for} \; \; \hat{\alpha} \in (P(\rho_-), P(\rho_+)) \label{change-Pbalance}
\end{align}
due to the assumptions (A1)--(A4). 
We note that \eqref{change-Pbalance} follows from the assumptions (A3) and (A4).
Then, a standard theory as in \cite[Lemma 2.6.1]{SCHAUBECK} can be applied to conclude that \eqref{2-3} admits a unique smooth solution and $\bar v$ satisfies the properties in (2). 
Therefore, the unique existence of the solution of \eqref{pro-translayer} and the properties in (1) can be proved by using the assumption (A1). 
\end{proof}

We will apply the exponential decay estimates \eqref{ex-decay-u} and \eqref{ex-decay-v}, and thus we hereafter fix an arbitrary 
\begin{equation}\label{def-gamma} 
\gamma \in (0, \min\{\gamma_{\max, \bar u}, \gamma_{\max, \bar v} \}). 
\end{equation}
We also use the linear operator $\L$ defined by
\begin{equation}\label{def-L}
L_{\bar u}\psi(\xi)=\left[ 2\sigma(\bar u (\xi))\psi'(\xi) \right]' - \left[ B(\bar u(\xi)) + D(\bar u(\xi))\right]\psi (\xi) 
\end{equation}
for a suitable $\psi: \R \to \R$. For the precise definition of $\L$, see the appendix.
Let $\nu$ be the constant defined by
\begin{align}\label{2-7}
\nu= \lan \bar v', (-L_{\bar u})\bar v'\ran_{L^2}/2,
\end{align}
where $\lan\cdot, \cdot\ran_{L^2}$ denotes the standard $L^2$-norm on $\R$.
We also define the constants $\theta_1, \theta_2$ by
\begin{align}\label{2-8}
\theta_1=\int_{\rho_-}^{\rho_+}\sqrt{2\tilde W(\rho)} \; d\rho, \quad 
\theta_2=\int_{\rho_-}^{\rho_+}P'(\rho)\sqrt{2\tilde W(\rho)} \; d\rho,
\end{align}
where the function $\tilde W$ is defined as
\begin{equation}\label{def-tilde-W}
\tilde W(\rho)=\int_{\rho_-}^\rho W'(\tilde \rho)P'(\tilde \rho) \; d\tilde \rho.
\end{equation}
We will show $\lan\bar u', \bar v' \ran_{L^2} =\theta_1, \lan \bar v', \bar v'\ran_{L^2}=\theta_2$ (cf. \eqref{inner-uv}).
Then, the mobility $\mu$ and the transport coefficient $\theta$ can be chosen as 
\begin{equation}\label{def-mu-theta} 
\mu:=\nu/\theta_1^2, \quad \theta:=\theta_2/\theta_1, 
\end{equation}
respectively.

In order to analyze the function $S_\e$ defined by \eqref{def-S}, we will use the representation of it by $J^H_\e$.
Since $S_\e(\phi)$ is defined as the supremum of $J^H_\e(\phi)$ in $H$, the existence of the maximizer $H$ of $J^H_\e(\phi)$ and some properties of the maximizer as stated in the following proposition will be required.

\begin{proposition}
Let $\phi$ be a function in $C^{2,3}([0,T]\times \T^d)$ such that
$0<\phi(t,x)<1$ for any $t\in[0,T]$ and $x\in \T^d$. 
Then, there exists a unique solution $H_{\max, \e}=H_{\max, \e}(\phi)\in C^{1,2}([0,T]\times\T^d)$ to the nonlinear Poisson equation
\begin{align}\label{2-9}
\partial_t\phi+\nabla\cdot[2\sigma(\phi)\nabla H_{\max, \e}]=\Delta P(\phi)+\dfrac{B(\phi)e^{H_{\max, \e}} - D(\phi)e^{-H_{\max, \e}}}{\e^2}.
\end{align}
Moreover, the supremum in \eqref{def-J} is achieved for $H_{\max, \e}$ and
\begin{equation}\label{2-5}
\begin{aligned}
S_\e(\phi)&=\dfrac{1}{\e} \int_0^T \int_{\T^d} \sigma(\phi)|\nabla H_{\max, \e}|^2 \; dxdt \\
& \quad + \dfrac{1}{\e^3} \int_0^T \int_{\T^d} B(\phi)\left(1-e^{H_{\max, \e}}+H_{\max, \e}e^{H_{\max, \e}} \right) \; dxdt \\
& \qquad + \dfrac{1}{\e^3} \int_0^T \int_{\T^d} D(\phi) \left(1-e^{-H_{\max, \e}}-H_{\max, \e}e^{-H_{\max, \e}}\right) \; dxdt.
\end{aligned}
\end{equation}
\end{proposition}

When $P(\rho)=\rho$ for any $\rho\in[0,1]$,
the (parabolic re-scaling $(t,x) \mapsto (\e^2 t, \e x)$ version of) proposition has been established in \cite[Lemma 2.1]{JLV}.
The general case can be deduced by performing the transformation $P(\phi)=\psi$.
 
Finally, we define the regularized version of a signed distance function stated in the introduction as follows. 
For a family of oriented smooth hyper-surfaces $\Gamma=\{\Gamma_t\}_{t\in[0,T]}$ with $\Gamma_t=\partial\Omega_t$ for some open $\Omega_t\subset \T^d$ and with the finite surface area for any $t \in [0,T]$,  
there exists a positive constant $\kappa$ satisfying 
\begin{equation}\label{def-K} 
\kappa \le 1/\left(\max_{t \in [0,T], x \in \Gamma_t} \max \{|\kappa_{t, 1}(x)|, |\kappa_{t, 2}(x)|, \ldots, |\kappa_{t, d-1}(x)|\}\right), 
\end{equation}
where $\kappa_{t, 1}, \kappa_{t, 2}, \ldots, \kappa_{t, d-1}$ are the principal curvatures of $\Gamma_t$, such that any point $(t, x) \in [0,T] \times \T^d$ satisfying ${\rm dist}(x, \Gamma_t) \le \kappa$ has a unique point $\zeta(t,x) \in \Gamma_t$ satisfying 
\[ {\rm dist}(x, \Gamma_t) = |x - \zeta(t,x)|. \]
We refer to \cite[Section 14.6]{GT} for the existence of such constant $\kappa$ and some properties of the distance function are summarized in Proposition \ref{prop:dist}.
We further choose $d(\cdot, t)$ as a regularized version of the signed distance from $\Gamma_t$ satisfying 
\begin{equation}\label{def-dist} 
d(t,x) = 
\begin{cases}
{\rm dist}(x, \Gamma_t) & \text{if} \; \; x \not\in \Omega_t \; \; \text{and} \; \; {\rm dist}(x, \Gamma_t) \le \kappa, \\
-{\rm dist}(x, \Gamma_t) & \text{if} \; \; x \in \Omega_t \; \; \text{and} \; \; {\rm dist}(x, \Gamma_t) \le \kappa. 
\end{cases}
\end{equation}
Hereafter, let $n_t$ denote the outward pointing unit normal vector of the surface $\Gamma_t$. 

\section{Proof of Theorem \ref{thm2-1}}\label{sec3}
In this section, we prove Theorem \ref{thm2-1}. 
Our proof is essentially based on the strategy in Bertini et al.\ \cite{BBP} and can be summarized as follows. 
Recall the expression of $\phi_\e$ from \eqref{2-6} and we here discuss the case when $R_\e = 0$ for simplicity. 
To compute the limit of $S_\e(\phi_\e)$ as $\e \to 0$, our first purpose is to calculate the power series expansion of $S_\e(\phi_\e)$ in $\e$, namely, to decompose $S_\e(\phi_\e)$ as the following form: 
\begin{equation}\label{form-dec-S} 
S_\e(\phi_\e) = \sum_{k \in \Z} \e^k \int_0^T \int_{\T^d} \phi^k_Q \left(t,x,\frac{d(t,x)}{\e}\right) \; dxdt, 
\end{equation}
where $\phi^k_Q$ is a function depending on $Q$. 
A key tool to obtain this kind of decomposition of $S_\e(\phi_\e)$ is the decomposition of the maximizer $H_{\max, \e}$ as 
\begin{equation}\label{decompose-H} 
H_{\max, \e}(t,x)=\e \hat{H}_1(t,x,d(t,x)/\e)+ \e^2 \hat{K}_\e (t,x),  
\end{equation}
where $\hat H_1$ is a unique solution to a linearized problem of \eqref{2-9} and is determined by the function $Q$ appeared in the choice of $\phi_\e$ (the linearized problem will be introduced around \eqref{linear-H}). 
We then apply the Taylor expansion for the integrands of $S_\e(\phi_\e)$ to conclude that, concerning the form \eqref{form-dec-S}; (i) $S_\e (\phi_\e)$ consists of terms with the coefficient $\e^k$ with $k \ge -1$; (ii) as $\e \to 0$, the term with coefficient $\e^{-1}$ is of constant order and converges to the iterated integral of $\phi^{-1}_Q (t,x,\xi)$ along $t \in [0,T]$, $x \in \Gamma_t$ and $\xi \in \R$; (iii) the other terms vanish as $\e \to 0$.
We note that the growth rate assumption of $Q$ with respect to $\xi$ as in \eqref{as-Q} is needed to obtain the exponential decay of $\hat{H}_1(t,x,\xi)$ with respect to $\xi$, which yields the convergence results (ii) and (iii). 
Since the minimum value of the integral of $\phi^{-1}_Q$ with respect to $Q$ coincides with $S_\ac (\Gamma)$, we can obtain the lower semi-continuity as stated in Theorem \ref{thm2-1}. 
In the proof of (2) in Theorem \ref{thm2-1}, we have to show that the minimizer $Q_{\min}$ of the integral of $\phi^{-1}_Q$ satisfies the growth rate property \eqref{as-Q}, and we complete the proof by letting $Q$ be the minimizer. 

In the following subsections, we first prepare the Taylor expansion for some functions that appeared in $S_\e(\phi_\e)$ and the decomposition properties of $H_{\max, \e}$ in Subsection \ref{sec:pre}, which will be used to obtain $\phi^k_Q$ in \eqref{form-dec-S}. 
The minimizing problem of the integral of $\phi^{-1}_Q$ is discussed in Subsection \ref{sec:mini}. 
We finally prove (1) and (2) of Theorem \ref{thm2-1} in Subsection \ref{sec3-2} and Subsection \ref{sec3-3}, respectively. 

Throughout the paper, for simplicity, we adopt the notations 
\begin{align*}
&d=d(t,x), \quad d_\e=d/\e, \quad Q_{d_\e}=Q(t,x,d_\e), \\
&\partial_\xi Q_{d_\e}=\partial_\xi Q(t,x,d_\e), \quad \partial_\xi^2 Q_{d_\e}=\partial_\xi^2 Q(t,x,d_\e) 
\end{align*}
for a smooth function $Q:[0,T]\times\T^d\times\R\to\R$ and also to distinguish the meaning of the derivatives with respect to $x$ we adopt the notions 
\begin{align*} 
&\nabla Q_{d_\e} = \nabla \big(Q(t,x,d_\e)\big), \quad \nabla_x Q_{d_\e} = \nabla Q \lfloor_{\xi = d_\e} = (\nabla Q)(t, x, d_\e), \\
&\Delta Q_{d_\e} = \Delta \big(Q(t,x,d_\e)\big), \quad \Delta_x Q_{d_\e} = (\Delta Q)(t,x,d_\e). 
\end{align*}
These notions are adopted similarly for a smooth function $H_1: [0,T] \times \T^d \times \R \to \R$, for example, $H_{1, d_\e} = H_1(t,x,d_\e)$. 
Moreover, we denote a generic element of $d_\e$ by $\xi$ and note that the operator $L_{\u}$ (or its inverse) always acts on a $\xi$ variable.

\subsection{Preliminary computations}\label{sec:pre}
 
In this subsection, we first list the Taylor expansion for key functions to obtain the Taylor expansion for $S_\e (\phi_\e)$ in $\e$. 
For simplicity of the notations, let $O_{F}(\e^\alpha)$ represent a function $G_\e: [0,T] \times \T^d\to\R$ depending on a function $F: [0,T] \times \T^d \to \R$ and satisfying 
\[ \varlimsup_{\e \to 0} \left(\sup_{(t,x) \in [0,T] \times \T^d} |G_\e(t,x)|/ \e^\alpha \right) < \infty \]
and 
\[ G_\e(t,x) = 0 \quad \text{for} \; \; (t,x) \in [0,T] \times \T^d, \; \; \e > 0 \quad \text{if} \; \; F \equiv 0, \]
where $\alpha \in \R$. 
 
\begin{proposition}\label{prop3-1}
Let $\gamma > 0$ be a constant satisfying \eqref{def-gamma}. 
Let $\phi_\e$ be the function defined by \eqref{2-6} for smooth functions $Q: [0,T]\times\T^d\times\R\to\R$ and $R_\e: [0,T]\times\T^d\to\R$ satisfying \eqref{as-Q} and \eqref{as-R}.
Then, it holds that 
\begin{align}
&\phi_\e=\u(d_\e)+\e\u'(d_\e)Q_{d_\e}+\e R_\e+\e^2 \mc R_\e^{(1)}(t, x, d_\e), \label{tay1}\\
&\partial_t\phi_\e=\u'(d_\e)\dfrac{\partial_t d}{\e}+\e\partial_t R_\e + \mc R_\e^{(2)}(t,x,d_\e), \label{tay2}\\
&\begin{aligned}
\Delta \left(P(\phi_\e) \right)
=&\; \dfrac{1}{\e^2}\v''(d_\e) +\dfrac{1}{\e}\{\v'''(d_\e)Q_{d_\e}+2\v''(d_\e)\partial_\xi Q_{d_\e}+\v'(d_\e)\partial_\xi^2 Q_{d_\e}\}\\
&\; +\frac{\Delta d}{\e} \v'(d_\e) +\dfrac{R_\e}{\e} \left(P'''(\u(d_\e)) (\u'(d_\e))^2 + P''(\u(d_\e)) \u''(d_\e)\right) \\
&\; +\mc R_\e^{(3)}(t, x, d_\e) + O_{R_\e}(\e^2), 
\end{aligned} \label{tay3}\\
&\begin{aligned}
B(\phi_\e)=&\; B\left(\u(d_\e)\right)+\e B'\left(\u(d_\e)\right)\u'(d_\e)Q_{d_\e}+\e B'\left(\u(d_\e)\right)R_\e\\
&\; +\e^2 \mc R_\e^{(5)}(t, x, d_\e) + O_{R_\e}(\e^2), 
\end{aligned}\label{tay4}\\
&\begin{aligned}
D(\phi_\e)=&\; D\left(\u(d_\e)\right)+\e D'\left(\u(d_\e)\right)\u'(d_\e)Q_{d_\e}+\e D'\left(\u(d_\e)\right)R_\e\\
&\; +\e^2 \mc R_\e^{(6)}(t,x,d_\e) +O_{R_\e}(\e^2), 
\end{aligned} \label{tay5}\\
&\sigma(\phi_\e) = \sigma(\u(d_\e)) + \e \mc R_\e^{(4)} (t,x,d_\e) +O_{R_\e}(\e), \label{tay-sigma}
\end{align}
where the remainders $\mc R_\e^{(i)}: [0,T] \times \T^d \times \R \to \R, i=1,2,\ldots,6,$ satisfy
\begin{align*}
\varlimsup_{\e \to 0} \sup_{(t,x, \xi) \in [0,T] \times \T^d \times \R} e^{\gamma |\xi|/2} |\mc R_\e^{(i)}(t,x,\xi)| < \infty. 
\end{align*}
Furthermore, letting $H_1: [0,T]\times \T^d \times \R \to \R$ and $K_\e: [0,T] \times \T^d \to \R$ be smooth functions satisfying 
\begin{align}
&\sup_{(t,x,\xi) \in [0,T] \times \T^d \times \R} \left(|H_1| + |\nabla H_1| + |\partial_\xi H_1| \right)e^{\tilde{\gamma} |\xi|} < \infty, \label{as-H}\\
&\varlimsup_{\e\to0} \sup_{(t,x) \in [0,T] \times \T^d} \left(|K_\e| + \e |\nabla K_\e| \right) < \infty \label{as-K}
\end{align}
for some $0 < \tilde{\gamma} < \gamma$, and defining $H$ as 
\begin{align}\label{tay20}
H(t,x)=\e H_1(t,x,d(t,x)/\e)+\e^2 K_\e(t,x),
\end{align}
it holds that 
\begin{align}
&\begin{aligned}
e^H=&\; 1+\e H_{1, d_\e} +\e^2(H_{1, d_\e}^2/2+K_\e) + \e^3 \mc R_\e^{(7)}(t,x, d_\e) +O_{K_\e}(\e^4),
\end{aligned} \label{tay6}\\
&\begin{aligned}
e^{-H}=&\; 1-\e H_{1, d_\e}+\e^2(H_{1, d_\e}^2/2-K_\e) + \e^3 \mc R_\e^{(8)}(t,x,d_\e) +O_{K_\e}(\e^4), 
\end{aligned} \label{tay7}\\
&\nabla H = \e \nabla_x H_{1, d_\e} + \partial_\xi H_{1, d_\e}\nabla d+\e^2 \nabla K_\e, \label{tay8}
\end{align}
where the remainders $\mc R_\e^{(i)}: [0,T] \times \T^d \times \R \to \R, i=7,8,$ satisfy
\begin{align*}
\varlimsup_{\e \to 0} \sup_{(t,x, \xi) \in [0,T] \times \T^d \times \R} e^{\tilde{\gamma} |\xi|/2} |\mc R_\e^{(i)}(t,x,\xi)| < \infty. 
\end{align*}
\end{proposition}

\begin{proof}
This proposition is a direct consequence of a straightforward computation using a Taylor expansion.
Therefore we prove \eqref{tay1} and \eqref{tay3} only and leave the other expansions to the reader.

We first prove \eqref{tay1}. The Taylor expansion for $\u$ at $d_\e$ shows
\begin{align*}
\phi_\e=\u(d_\e)+\e\u'(d_\e)Q_{d_\e}+\e R_\e+\frac{\e^2}{2}\u''\left(d_\e+\eta_\e \e Q_{d_\e}\right)Q_{d_\e}^2,
\end{align*}
for some $0\le \eta_\e\le1$.
Let $\mc R_\e^{(1)}=\u''\left(\xi+\eta_\e \e Q\right)Q^2/2$.
It remains to show that there exists a constant $C_0$, such that
\begin{align*}
|\mc R_\e^{(1)}(t,x,\xi)|\le C_0e^{-\gamma|\xi|/2},\quad (t,x,\xi)\in[0,T]\times\T^d\times\R.
\end{align*}
Indeed, it follows from \eqref{ex-decay-u} and the assumption $|Q(t,x,\xi)| \le C (1+|\xi|)$ for some $C>0$ that, for $\e > 0$ small, 
\begin{align*}
|\mc R_\e^{(1)}(t,x,\xi)|\le&\; \frac{C}{2} (1+|\xi|)^2 e^{-\gamma |\xi + \eta_\e \e Q(t, x, \xi)|} \le \frac{C}{2} (1+|\xi|^2) e^{- \gamma (|\xi| - \e C (1+|\xi|))} \\
\le&\; \frac{C}{2} (1 + |\xi|)^2 e^{-3\gamma |\xi|/4} \le C_0 e^{-\gamma |\xi|/2} 
\end{align*}
for some $C_0>0$. 
Therefore \eqref{tay1} holds.

We next prove \eqref{tay3}. 
We have by a simple calculation 
\begin{equation}\label{cal-tay1}
\begin{aligned}
&\; \Delta P(\phi_\e) = \Delta P(\u(d_\e + \e Q_{d_\e}) + \e R_\e) \\
=&\; \nabla \cdot \left\{P'(\u(d_\e + \e Q_{d_\e}) + \e R_\e) \left( \u'(d_\e+\e Q_{d_\e}) \left(\frac{\nabla d}{\e} + \e \nabla_x Q_{d_\e} + \partial_\xi Q_{d_\e} \nabla d \right) + \e \nabla R_\e \right)\right\} \\
=&\; P''(\u(d_\e + \e Q_{d_\e}) + \e R_\e)\times \\
&\; \quad \Bigg\{ (\u'(d_\e + \e Q_{d_\e}))^2 \left| \frac{\nabla d}{\e} + \e \nabla_x Q_{d_\e} + \partial_\xi Q_{d_\e} \nabla d \right|^2 \\
&\; \quad \; + 2 \u'(d_\e + \e Q_{d_\e}) \left\langle  \frac{\nabla d}{\e} + \e \nabla_x Q_{d_\e} + \partial_\xi Q_{d_\e} \nabla d, \e \nabla R_\e \right\rangle_{\T^d} + \e^2 |\nabla R_\e|^2 \Bigg\} \\
&\; + P'(\u(d_\e + \e Q_{d_\e}) + \e R_\e)\times \\
&\; \quad \Bigg\{ \u''(d_\e + \e Q_{d_\e}) \left| \frac{\nabla d}{\e} + \e \nabla_x Q_{d_\e} + \partial_\xi Q_{d_\e} \nabla d \right|^2 \\
&\; \quad \; + \u'(d_\e + \e Q_{d_\e}) \Bigg(\frac{\Delta d}{\e} + \e \Delta_x Q_\e + 2 \langle \nabla_x \partial_\xi Q_{d_\e}, \nabla d \rangle_{\T^d} + \frac{1}{\e} \partial_\xi^2 Q_{d_\e} |\nabla d|^2 + \partial_\xi Q_{d_\e} \Delta d\Bigg) + \e \Delta R_\e \Bigg\},
\end{aligned}
\end{equation}
where $\lan\cdot, \cdot\ran_{\T^d}$ denotes the standard inner product in $\T^d$.
We hereafter write $F_\e(t,x,d_\e)$ uniformly as $\tilde{\mc R}_\e$ for any function $F_\e: [0,T] \times \T^d \times \R \to \R$ satisfying 
\[ \varlimsup_{\e \to 0} \sup_{(t,x,\xi) \in [0,T] \times \T^d \times \R} e^{\gamma' |\xi|} |F_\e(t,x,\xi)| < \infty \]
for any $\gamma' \in (0, \gamma)$ and let $\tilde{\mc R}_{\e, d_\e}$ represent $\tilde{\mc R}_\e(t,x, d_\e)$. 
Since the Taylor expansion shows that 
\begin{align*}
&\; P''(\u(d_\e + \e Q_{d_\e}) + \e R_\e) \\
=&\; P''(\u(d_\e) + \e \u'(d_\e)Q_{d_\e} + \e^2 \tilde{\mc R}_{\e, d_\e} + \e R_\e ) \\
=&\; P''(\u(d_\e)) + P'''(\u(d_\e)) (\e \u'(d_\e) Q_{d_\e} + \e^2 \tilde{\mc R}_{\e, d_\e} + \e R_\e) + \e^2 \tilde{\mc R}_{\e, d_\e} + O_{R_\e}(\e^2),  
\end{align*}
the first term of \eqref{cal-tay1} can be computed as
\begin{align*}
&\; P''(\u(d_\e)) \Bigg\{ (\u'(d_\e))^2  \left(\frac{|\nabla d|^2}{\e^2} + \frac{2 \partial_\xi Q_{d_\e}}{\e} |\nabla d|^2\right) + \frac{2\u'(d_\e) \u''(d_\e) Q_{d_\e}}{\e} |\nabla d|^2 + \tilde{\mc R}_{\e, d_\e} + O_{R_\e}(\e^2) \Bigg\} \\
&\; +P'''(\u(d_\e)) \Big(\e \u'(d_\e) Q_{d_\e} + \e^2 \tilde{R}_{\e, d_\e} + \e R_\e\Big) \left(\frac{(\u'(d_\e))^2}{\e^2} |\nabla d|^2 + \frac{1}{\e} \tilde{\mc R}_{\e, d_\e} + O_{R_\e}(\e^2)\right) \\
&\; + \tilde{\mc R}_{\e, d_\e} + O_{R_\e}(\e^2) \\
=&\; \frac{|\nabla d|^2}{\e^2} P''(\u(d_e)) (\u'(d_\e))^2 + \frac{Q_{d_\e}|\nabla d|^2}{\e}\Big(2P''(\u(d_\e)) \u'(d_\e)\u''(d_\e) + P'''(\u(d_\e))(\u'(d_\e))^3 \Big) \\
&\; + \frac{\partial_\xi Q_{d_\e}|\nabla d|^2}{\e} P''(\u(d_\e)) (\u'(d_\e))^2 + \frac{R_\e |\nabla d|^2}{\e} P'''(\u(d_\e))(\u'(d_\e))^2 + \tilde{\mc R}_{\e, d_\e} + O_{R_\e}(\e^2). 
\end{align*}
We can similarly compute the second term of \eqref{cal-tay1} as
\begin{align*}
&\; \frac{|\nabla d|^2}{\e^2} P'(\u(d_\e))\u''(d_\e) + \frac{Q_{d_\e}|\nabla d|^2}{\e} \Big(P'(u_\e)\u'''(d_\e) + P''(\u(d_\e))\u'(d_\e)\u''(d_\e)\Big) \\
&\; + \frac{\partial_\xi Q_{d_\e} |\nabla d|^2}{\e} P'(\u(d_\e)) \u''(d_\e) + \frac{\partial_\xi^2 Q_{d_\e}|\nabla d|^2}{\e} P'(\u(d_\e)) \u'(d_\e) \\
&\; + \frac{\Delta d}{\e} P'(\u(d_\e)) \u'(d_\e) + \frac{R_\e |\nabla d|^2}{\e} P''(\u(d_\e)) \u''(d_\e) + \tilde{\mc R}_{\e, d_\e} + O_{R_\e}(\e^2). 
\end{align*}
We here note that Proposition \ref{prop:dist} shows $|\nabla d(x,t)| = 1$ if $|{\rm dist}(x, \Gamma_t)| \le \kappa$ and, for any function $w: [0,T] \times \T^d \times \R \to \R$ satisfying
\[ C_w :=\sup_{(t, x, \xi) \in [0,T] \times \T^d \times \R}|w(t, x, \xi)|e^{-\gamma |\xi|} < \infty, \]
the exponential decay yields, for any $0 < \gamma' < \gamma$, 
\[ |w(t, x, d_\e)| \le C_w e^{-(\gamma-\gamma')\kappa/\e}e^{-\gamma' d_\e} \quad \text{if} \; \; |{\rm dist}(x, \Gamma_t)| \ge \kappa. \] 
Therefore, for any $k\in \N$, it holds that 
\begin{equation}\label{pr-d} 
\frac{|\nabla d|^2 w(t, x, d_\e)}{\e^k} \le \frac{w(t, x, d_\e)}{\e^k} + \tilde{\mc R}_{\e, d_\e} \quad \text{for} \; \; (t,x) \in [0,T] \times \T^d. 
\end{equation}
Substituting these inequalities into \eqref{cal-tay1}, we have \eqref{tay3} due to $v(\xi) = P(\u(\xi))$. 
\end{proof}

We next introduce a linearized problem of \eqref{2-9} to decompose the unique solution $H_{\max, \e}$ of \eqref{2-9}.
We start from a formal expansion for the solution to \eqref{2-9}.
Let $\phi$ be chosen as $\phi_\e$ given by \eqref{2-6} with $R_\e = 0$.
Note that, when $H$ has the form \eqref{tay20}, a Taylor expansion shows
\begin{align*}
\nabla \cdot (\sigma(\phi_\e) \nabla H) = \frac{1}{\e} \left( \sigma'(\u(d_\e)) \u'(d_\e) \partial_\xi H_{1, d_\e} + \sigma(\u(d_\e)) \partial_\xi^2 H_{1, d_\e}\right) + O(1)
\end{align*}
at $x$ near the moving surface $\Gamma_t$. 
Using the ansatz \eqref{decompose-H}, the Tayler expansions as in Proposition \ref{prop3-1} and the above expansion yields formally, as $\e \to 0$,  
\begin{align*} 
&\; \partial_t \phi_\e + \nabla \cdot (2 \sigma(\phi_\e) \nabla H_{\max, \e}) \\
=&\; \frac{1}{\e} \Big(\u'(d_\e) \partial_t d + 2\{\sigma'(\u(d_\e))\u'(d_\e) \partial_\xi \hat{H}_{1,d_\e} + \sigma (\u(d_\e)) \partial_\xi^2 \hat{H}_{1, d_\e}\} \Big) + O(1) 
\end{align*}
and 
\begin{align*} 
&\; \Delta P(\phi_\e) + \frac{B(\phi_\e) e^{H_{\max, \e}} - D(\phi_\e) e^{-H_{\max, \e}}}{\e^2} \\
=&\; \frac{1}{\e^2} \Big( \v''(d_\e) + B(\u(d_\e)) - D(\u(d_\e)) \Big) \\
&\; + \frac{1}{\e} \Big( \v'''(d_\e) + \{B'(\u(d_\e))-D'(\u(d_\e))\} \u'(d_\e) \Big) Q_\e \\
&\; + \frac{1}{\e} \Big(\v'(d_\e) \Delta d + 2 \v''(d_\e) Q_\e' + \v'(d_\e) Q_\e'' + \{B(\u(d_\e)) + D(\u(d_\e))\} \hat{H}_{1,\e} \Big) + O(1). 
\end{align*}
Since it follows from \eqref{pro-translayer} or \eqref{2-3} that the first and second terms in the right-hand side of the latter equality are zero, letting $d_\e = \xi$, we obtain the $1/\e$ order equality in \eqref{2-9} as 
\begin{equation}\label{linear-H} 
L_{\u} \hat{H}_1 (t,x,\xi) = \v'(\xi) \Delta d(t,x) + 2 \v''(\xi) \partial_\xi Q(t,x,\xi) + \v'(\xi) \partial_\xi^2 Q(t,x,\xi) - \u'(\xi) \partial_t d(t,x),  
\end{equation}
where $\L$ is defined by \eqref{def-L}. 
The solution $\hat{H}_1:[0,T] \times \T^d \times \R \to \R$ of the linearized equation \eqref{linear-H} can be defined since $\L$ has the inverse (see also Proposition \ref{prop:Lu}).

In the following, we rigorously perform the decomposition for $H_{\max, \e}$ in terms of $\e \hat{H}_{1, \e} = \e \hat{H}_1 (t, x, d_\e)$ and an error function $\e^2 \hat{K}_\e$. We also discuss estimates for $\hat H_{1,\e}$ and $\hat K_\e$, which will be needed to calculate the limit of $S_\e (\phi_\e)$. 

\begin{proposition}\label{prop:dec-H}
Let $Q: [0,T] \times \T^d \times \R \to \R$ be a smooth function satisfying \eqref{as-Q}. 
Define $\phi_\e: [0,T] \times \T^d\to[0,1]$ by \eqref{2-6} with $R_\e=0$. 
Let $H_{\max, \e}: [0,T] \times \T^d \to \R$ and $\hat{H}_1: [0,T] \times \T^d \times \R \to \R$ be the solution of \eqref{2-9} and \eqref{linear-H}, respectively. 
Define the function $\hat K_\e: [0,T] \times \T^d \to \R$ through the decomposition 
\[ H_{\max, \e}(t,x) = \e \hat{H}_1 (t,x, d_\e) + \e^2 \hat{K}_\e(t,x) \quad \text{for} \; \; (t, x) \in [0,T] \times \T^d. \]
Then, letting $\gamma > 0$ be satisfying \eqref{def-gamma}, there exists $0 < \tilde{\gamma} < \gamma$ such that  
\begin{align}
&\sup_{(t,x,\xi) \in [0,T] \times \T^d \times \R} e^{\tilde{\gamma} |\xi|} \sum_{i=0}^2 \sum_{j=0}^2 |\partial_\xi^i \nabla^j \hat{H}_1|< \infty, \label{es-H*}\\
&\varlimsup_{\e\to0}\left\{\sup_{(t, x) \in [0,T] \times \T^d} |\hat{K}_\e| + \e|\nabla \hat{K}_\e| \right\} < \infty. \label{es-He}
\end{align}
\end{proposition}

\begin{proof}
We first note that $\hat{H}_1$ is well-defined by Proposition \ref{prop:ele}. 
The estimate \eqref{es-H*} also follows from Proposition \ref{prop:Lu}. 
Indeed, due to the assumption \eqref{as-Q} and the exponential decay estimates \eqref{ex-decay-u} and \eqref{ex-decay-v}, we can see the exponential decay of the right-hand side of \eqref{linear-H} as 
\[ \sup_{(t, x, \xi) \in [0,T] \times \T^d \times \R} e^{\tilde{\gamma} |\xi|} |\v' \Delta d + 2 \v'' \partial_\xi Q + \v' \partial_\xi^2 Q - \u' \partial_t d| < \infty \]
for any $0 < \tilde{\gamma} < \gamma$. 
Therefore, choosing $\tilde{\gamma}$ so that $\tilde{\gamma} \le \gamma_0 / 2$, where $\gamma_0$ is the constant obtained in Proposition \ref{prop:Lu}, we have 
\[ \sup_{(t, x, \xi) \in [0,T] \times \T^d \times \R} e^{\tilde{\gamma} |\xi|} \sum_{i=0}^2 |\partial_\xi^i \hat{H}_1| < \infty. \]
For the estimate of the derivative $\nabla \hat{H}_1$, we denote by $e_k$ the unit vector facing the $x_k$-axis for $k=1,2, \ldots, d$, and we consider the deviation of $\hat{H}_1$ in the direction $e_k$ as follows: for $h\in\R$,
\begin{align*} 
&f(t,x,\xi; h) := L_{\u}(\hat{H}_1(t, x+he_k, \xi) - \hat{H}_1(t, x, \xi)) \\
&\quad= \v'(\xi)\Big(\Delta d(t, x+he_k) - \Delta d(t,x)\Big) + 2\v''(\xi)\Big(\partial_\xi Q(t, x+he_k, \xi) - \partial_\xi Q(t, x, \xi)\Big) \\
&\qquad +2 \v'(\xi)\Big(\partial_\xi^2 Q(t, x+he_k, \xi) - \partial_\xi^2 Q(t,x,\xi)\Big) - \u''(\xi)\Big(\partial_t d(t, x+he_k) - \partial_t d(t,x)\Big).
\end{align*}
Since $d$ is smooth and $Q$ satisfies \eqref{as-Q}, we have by the exponential decay estimates \eqref{ex-decay-u} and \eqref{ex-decay-v} 
\[ \sup_{(t,x,\xi) \in [0,T] \times \T^d \times \R} e^{\tilde{\gamma}|\xi|} |f(t, x, \xi; h)| \le C|h| \]
for some positive constant $C>0$ and any $h \in \R$. 
Therefore, applying Proposition \ref{prop:Lu}, we have 
\[ \sup_{(t, x, \xi) \in [0,T] \times \T^d \times \R} e^{\tilde{\gamma}|\xi|} \sum_{i=0}^2 \frac{|\partial_\xi^i \hat{H}_1(t, x+he_k, \xi) - \partial_\xi^i \hat{H}_1(t, x, \xi)|}{|h|} \le \tilde{C} \]
for some positive constant $\tilde C>0$ and any $h \neq 0$. 
Letting $h \to 0$ and taking summation with respect to $k$, we obtain 
\[ \sup_{(t, x, \xi) \in [0,T] \times \T^d \times \R} e^{\tilde{\gamma}|\xi|} \sum_{i=0}^2 |\partial_\xi^i \nabla \hat{H}_1| < \infty. \]
A similar argument can be applied for the estimate of $\nabla^2 \hat{H}_1$ to obtain \eqref{es-H*}. 

We next consider the decomposition $H_{\max, \e}(t,x) = \mathcal{H}_\e(t,x) + \mathcal{K}_\e(t,x)$ with the smooth function $\mathcal{H}_\e: [0,T] \times \T^d\to\R$ solving 
\begin{equation}\label{eq-math-H} 
\nabla \cdot [2\sigma(\phi_\e) \nabla \mathcal{H}_\e] - \frac{B(\phi_\e) + D(\phi_\e)}{\e^2}\mathcal{H}_\e = -\partial_t \phi_\e + \Delta P(\phi_\e) + \frac{B(\phi_\e) - D(\phi_\e)}{\e^2} 
\end{equation}
and therefore $\mathcal{K}_\e: [0,T] \times \T^d\to\R$ satisfying 
\begin{equation}\label{eq-math-K} 
\nabla \cdot [2\sigma(\phi_\e) \nabla \mathcal{K}_\e] = \frac{B(\phi_\e) (e^{\mathcal{H}_\e + \mathcal{K}_\e} -1 - \mathcal{H}_\e) - D(\phi_\e) (e^{-\mathcal{H}_\e - \mathcal{K}_\e} - 1 + \mathcal{H}_\e)}{\e^2}. 
\end{equation}
Since $B(\phi_\e) + D(\phi_\e)$ and $\sigma(\phi_\e)$ are strictly positive, the first equation \eqref{eq-math-H} has a unique solution in $L^2(\T^d)$ for each $(t,\xi) \in [0,T] \times \R$, which is a smooth function of $(t,x,\xi)$ by the smoothness of $\phi_\e$ and the elliptic regularity (cf.\ \cite[Chapter 6]{EVANS}).
Applying \eqref{tay2}--\eqref{tay5} of Proposition \ref{prop3-1} and \eqref{2-3}, we can see that there exists $M>0$ such that 
\[ \sup_{(t,x)\in[0,T]\times\T^d} \left|-\partial_t \phi_\e + \Delta P(\phi_\e) + \frac{B(\phi_\e) - D(\phi_\e)}{\e^2}\right| \le M/\e \]
for $\e > 0$ small. 
Since $B$ and $D$ are positive, letting 
\[ c := \inf_{u \in [0,1]} \{B(u) + D(u)\} > 0, \]
we see that it holds that 
\[ \nabla \cdot [2 \sigma(\phi_\e) \nabla (\mathcal{H}_\e - M\e/c)] - \frac{B(\phi_\e) + D(\phi_\e)}{\e^2}(\mathcal{H}_\e -M\e/c) \ge 0. \]
Therefore, the maximum principle (see \cite[Theorem 3.5]{GT} for example) yields $\mathcal{H}_\e \le M\e/c$. 
Since $\mathcal{H}_\e \ge -M\e/c$ can be proved similarly, the function $\mathcal{H}_\e$ is $O(\e)$. 

Next, we prove that
\[ \nabla \cdot [ 2 \sigma(\phi_\e) \nabla (\mathcal{H}_\e - \e \hat{H}_{1, d_\e})] - \frac{B(\phi_\e) + D(\phi_\e)}{\e^2} (\mathcal{H}_\e - \e \hat{H}_{1, d_\e}) =O(1) \quad \text{as} \; \; \e \to 0. \]
To see this, using \eqref{pr-d}, we observe by the exponential decay estimates \eqref{ex-decay-u} and \eqref{es-H*} 
\begin{align*}
&\; \nabla \cdot [2 \sigma(\phi_\e) \nabla (\e \hat{H}_{1, d_\e})] - \frac{B(\phi_\e) + D(\phi_\e)}{\e^2} (\e \hat{H}_{1, d_\e})\\
=&\;  \frac{2}{\e} \left[\sigma'(\u(d_\e)) \partial_\xi \hat{H}_{1,d_\e} + \sigma(\u(d_\e)) \partial_\xi^2 \hat{H}_{1, d_\e} \right] - \frac{B(\phi_\e) + D(\phi_\e)}{\e} \hat{H}_{1, d_\e} + O(1). 
\end{align*}
We further apply \eqref{tay4} and \eqref{tay5} to obtain 
\[ \nabla \cdot [2 \sigma(\phi_\e) \nabla (\e \hat{H}_{1, d_\e})] - \frac{B(\phi_\e) + D(\phi_\e)}{\e^2} (\e \hat{H}_{1, d_\e})= \frac{1}{\e} L_{\u} \hat{H}_1 \lfloor_{\xi = d_\e} + O(1). \]
On the other hand, applying \eqref{tay2}--\eqref{tay5} and \eqref{pro-translayer} or \eqref{2-3} to \eqref{eq-math-H}, we have 
\begin{align*} 
&\; \nabla \cdot [2 \sigma(\phi_\e) \nabla \mathcal{H}_\e] - \frac{B(\phi_e) + D(\phi_\e)}{\e^2} \mathcal{H}_\e \\
=&\; \frac{1}{\e}\left(\v'(d_\e) \Delta d + 2\v''(d_\e) \partial_\xi Q_{d_\e} + \v'(d_\e) \partial_\xi^2 Q_{d_\e} - \u'(d_\e) \partial_t d\right) + O(1). 
\end{align*}
Therefore by \eqref{linear-H}, it holds that 
\begin{equation}\label{eq-order1} 
\nabla \cdot [2 \sigma(\phi_\e) \nabla (\mathcal{H}_\e - \e \hat{H}_{1, d_\e})] - \frac{B(\phi_\e) + D(\phi_\e)}{\e^2} (\mathcal{H}_\e - \e \hat{H}_{1, d_\e})=O(1).
\end{equation}
Due to a similar argument to obtain $\mathcal{H}_\e = O(\e)$, the maximum principle again yields $\mathcal H_\e -\e \hat H_{1,d_\e}=O(\e^2)$.

Representing $\hat K_\e$ as
\begin{align*}
\hat K_\e = \e^{-2}(H_{\max,\e} -\e \hat H_{1,d_\e})=\e^{-2}(\mc H_\e -\e \hat H_{1,d_\e}) + \e^{-2} \mc K_\e,
\end{align*}
in order to deduce the bound
\begin{align*}
\varlimsup_{\e\to0}\left\{\sup_{(t, x) \in [0,T] \times \T^d} |\hat{K}_\e| \right\} < \infty,
\end{align*}
it is enough to show that $\mc K_\e=O(\e^2)$. 
Recalling that $B(\phi_\e)$ and $D(\phi_\e)$ are strictly positive as assumed in {\rm (A2)} and noticing that the nonlinear term on the right-hand side of \eqref{eq-math-K} is an increasing function of $\mathcal{K}_\e$, by comparison principle it is enough to construct super- and sub-solutions of \eqref{eq-math-K} in the form $\mc K_\e^\pm=\pm C\e^2$. 
Here $C$ is a large constant determined later. 
We note that, due to $\mathcal{H}_\e = O(\e)$, we have by the Taylor expansion 
\[ e^{- \mathcal{H}_\e - \mc K_\e^+} \le e^{-\mc H_\e} = 1 - \mc H_\e + O(\e^2). \]
Therefore, by the elementary inequality $e^x\ge 1+x$ we have 
\begin{align*}
&\; \nabla \cdot [2\sigma(\phi_\e) \nabla \mathcal{K}_\e^+] - \frac{B(\phi_\e) (e^{\mathcal{H}_\e + \mathcal{K}_\e^+} -1 - \mathcal{H}_\e) - D(\phi_\e) (e^{-\mathcal{H}_\e - \mathcal{K}_\e^+} - 1 + \mathcal{H}_\e)}{\e^2}\\
\le&\;  -\frac{B(\phi_\e)}{\e^2} \mc K_\e^+ + O(1) = - B(\phi_\e) C + O(1) \le 0 
\end{align*}
if $C>0$ is chosen sufficiently large and $\e$ is small since $B(\phi_\e)$ is strictly positive. 
Similarly, we can prove that 
\[ \nabla \cdot [2 \sigma(\phi_\e) \nabla \mathcal{K}_\e^-] - \frac{B(\phi_\e) (e^{\mathcal{H}_\e + \mathcal{K}_\e^-} -1 - \mathcal{H}_\e) - D(\phi_\e) (e^{-\mathcal{H}_\e - \mathcal{K}_\e^-} - 1 + \mathcal{H}_\e)}{\e^2} \ge 0 \] 
for $C>0$ large and $\e > 0$ small. 
We thus obtain $\mc K_\e = O(\e^2)$ and $\hat{K}_\e = O(1)$. 

It remains to show
\begin{equation}\label{es-remain}
\varlimsup_{\e\to0}\left\{\sup_{(t, x) \in [0,T] \times \T^d} \e|\nabla\hat{K}_\e| \right\} < \infty.
\end{equation}
Since $\mc H_\e=O(\e)$ and $\mc K_\e=O(\e^2)$, by applying the Taylor expansion to $e^{\mathcal{H}_\e + \mathcal{K}_\e}$ and $e^{-\mathcal{H}_\e - \mathcal{K}_\e}$, it follows from \eqref{eq-math-K} that
\begin{equation}\label{eq-order2}
\nabla \cdot [\sigma(\phi_\e) \nabla \mathcal{K}_\e] - \frac{B(\phi_\e)  + D(\phi_\e) }{\e^2} \mc K_\e=O(1). 
\end{equation}
By \eqref{eq-order1} and \eqref{eq-order2}, we obtain 
\begin{align*}
\nabla \cdot [\sigma(\phi_\e) \nabla \hat{K}_\e] - \frac{B(\phi_\e)  + D(\phi_\e) }{\e^2} \hat{K}_\e=O(\e^{-2}).
\end{align*}
We here use the change of variables $y=\e x$ to define 
\[ \hat{K}_{\e, y}(t, y) := \hat{K}_\e (t, y/\e), \quad \phi_{\e, y}(t, y) := \phi_\e(t, y/\e). \]
We then obtain 
\[ \nabla_y \cdot [\sigma(\phi_{\e,y}) \nabla_y \hat{K}_{\e, y}] - [B(\phi_{\e, y})  + D(\phi_{\e, y})] \hat{K}_{\e, y}=O(1), \]
which yields by standard elliptic regularity estimates (cf.\ the estimate follows from \cite[Theorem 9.11]{GT} and the Sobolev inequality)
\[ \varlimsup_{\e \to 0} \left\{\sup_{(t,y) \in [0,T] \times \T^d_\e} |\nabla_y \hat{K}_{\e, y}|  \right\} < \infty \]
since $\hat{K}_{\e, y} = O(1)$. 
Using the change of variables $y=\e x$ again, we have \eqref{es-remain}. 
\end{proof}

\subsection{Minimizing problem} \label{sec:mini}

Let $\bar Q:[0,T]\times\T^d\times\R\to\R$ be a smooth function satisfying 
\begin{equation}\label{as-bQ}
\sup_{(t,x,\xi) \in [0,T] \times \T^d \times \R} \frac{|\bar Q| + |\partial_\xi \bar Q| + |\partial_\xi^2 \bar Q|^2}{1+|\xi|} < \infty 
\end{equation}
and define the function $F_{\bar Q}:[0,T]\times\T^d\times\R\to\R$ by
\begin{equation}\label{3-2}
F_{\bar Q}(t,x,\xi) =\u'(\xi) \partial_t d(t,x) - \v'(\xi) \Delta d(t,x) -2\v''(\xi) \partial_\xi \bar Q(t,x,\xi) - \v'(\xi) \partial_\xi^2 \bar Q(t,x,\xi).
\end{equation}
Note that, for the relation between $F_{\bar Q}$ and $\hat{H}_1$ defined by \eqref{linear-H}, it holds that $- L_{\u} \hat{H}_1 = F_Q$. 
In the proof of Theorem \ref{thm2-1}, we will need to minimize 
\begin{align}\label{3-3}
\int_\R F_{\bar Q}(-L_{\u})^{-1}F_{\bar Q} \; d\xi,
\end{align}
for each $(t,x)\in[0,T]\times\T^d$ in $\bar Q$.
Recall the definitions of $\mu, \theta$ from \eqref{2-8}, \eqref{def-mu-theta}.

\begin{proposition}\label{prop3-2}
Let $\bar Q: [0,T] \times \T^d \times \R$ be a smooth function satisfying \eqref{as-bQ}. 
Then, it holds that 
\begin{align*}
\int_\R F_{\bar Q}(-L_{\u})^{-1}F_{\bar Q} \; d\xi \ge
\dfrac{(\partial_t d - \theta\Delta d)^2}{2\mu} \quad \text{for} \; \; (t, x) \in [0,T] \times \T^d.
\end{align*}
Furthermore, a minimizer $\bar Q_{\min}$ is given by 
\begin{equation}\label{min-Q} 
\begin{aligned}
&\; \bar Q_{\min}(t,x,\xi) \\
=&\; \int_0^\xi \frac{1}{(\v')^2(\tilde{\xi})} \int_{-\infty}^{\tilde{\xi}} \left(\u'(\hat{\xi}) \partial_t d(t,x) - \v'(\hat{\xi}) \Delta d(t,x) - \frac{\lambda(t,x)}{2} L_{\u} \v'(\hat{\xi})\right) \v'(\hat{\xi}) \; d\hat{\xi} d\tilde{\xi}, 
\end{aligned}
\end{equation}
where $\lambda: [0, T] \times \T^d$ is a smooth function defined as 
\begin{equation}\label{min-lambda} 
\lambda(t,x) = \frac{2 (\|\v'\|_{L^2}^2 \Delta d(t,x) - \langle \u', \v'\rangle_{L^2} \partial_t d(t,x))}{ \langle -L_{\u} \v', \v'\rangle_{L^2}}, 
\end{equation}
and $\bar Q_{\min}$ satisfies \eqref{as-Q} replaced $Q$ by $\bar Q_{\min}$. 
\end{proposition}

\begin{proof}
Fix $(t,x) \in [0,T] \times \T^d$. 
We use $\bar Q'$ instead of $\partial_\xi \bar Q(t,x, \xi)$ and omit the variables $t, x$ for simplicity. 
Noticing $2\v'' \bar Q' + \v' \bar Q''$ is of class $L^2(\R)$ due to the exponential decay \eqref{ex-decay-v} and the growth rate assumption \eqref{as-bQ}, and is perpendicular with $\v'$ in $L^2(\R)$, we can re-formulate the minimizing problem as 
\begin{align*}
&\; \inf \left\{\int_\R F_{\bar Q}(-L_{\u})^{-1}F_{\bar Q} \; d\xi: \text{$\bar Q$ satisfies \eqref{as-bQ}} \right\} \\
\ge &\; \inf \left\{ \langle \u' \partial_t d - \v' \Delta d - \psi, (-L_{\u})^{-1}(\u' \partial_t d - \v' \Delta d - \psi) \rangle_{L^2} : \text{$\psi \in L^2(\R)$ s.t. $\psi \perp \v'$} \right\}, 
\end{align*}
where we denote $\psi \perp \phi$ for $\psi,\phi\in L^2(\R)$ if $\lan\psi, \phi\ran_{L^2}=0$.
We note that the equality holds if a minimizer $\psi_{\min}$ for the latter minimizing problem exists and a solution $\bar Q_{\min}$ to 
\begin{equation}\label{ODE-Q} 
2 \v'' \bar Q_{\min}' + \v' \bar Q_{\min}'' = \psi_{\min}
\end{equation} 
satisfies \eqref{as-bQ}; hence, we will solve the solution $\bar Q_{\min}$ and prove that $\bar Q_{\min}$ satisfies the stronger estimate \eqref{as-Q} than \eqref{as-bQ}. 
Due to this argument, we can also see that $\bar Q_{\min}$ is a minimizer of the original minimizing problem. 

We thus define functional 
\[ G(\psi) := \langle \u' \partial_t d - \v' \Delta d - \psi, (-L_{\u})^{-1}(\u' \partial_t d - \v' \Delta d - \psi) \rangle_{L^2} \quad \text{for} \; \; \psi \in L^2(\R) \]
and consider the minimizing problem 
\begin{equation}\label{mini-pro}
\inf_{\psi \in L^2: \psi \perp \v'} G(\psi). 
\end{equation}
Applying the method of Lagrange multiplier, we see that a minimizer $\psi_{\min} \in L^2(\R)$ of \eqref{mini-pro} satisfies 
\[ \langle \phi, (-L_{\u})^{-1}(\u' \partial_t d - \v' \Delta d - \psi_{\min}) \rangle_{L^2} + \langle \u' \partial_t d - \v' \Delta d - \psi_{\min}, (-L_{\u})^{-1} \phi \rangle_{L^2} = \lambda \langle \v', \phi \rangle_{L^2} \]
for any $\phi \in L^2(\R)$, where $\lambda$ is the Lagrange multiplier, if the minimizer exists. 
Since $L_{\u}$ is self-adjoint on $L^2(\R)$, it is equivalent to 
\[ \psi_{\min} = \u' \partial_t d - \v' \Delta d - \frac{\lambda}{2} L_{\u} \v'. \]
Therefore, the orthogonal condition $\psi_{\min} \perp \v'$ shows that $\lambda$ is given by \eqref{min-lambda} if the minimizer $\psi_{\min}$ exists. 
We next prove that $\psi_{\min}$ is a minimizer of \eqref{mini-pro}. 
For this purpose, note that $\lambda$ is chosen so that $\psi_{\min} \perp \v'$ holds.
Therefore it is enough to prove $G(\psi_{\min} + \psi) \ge G(\psi_{\min})$
for any function $\psi \in L^2(\R)$ with $\psi \perp \v'$. 
By direct calculations, we have 
\[ G(\psi_{\min}) = \frac{\lambda^2}{4} \langle -L_{\u} \v', \v' \rangle_{L^2} = \frac{(\partial_t d \langle \u', \v'\rangle_{L^2} - \Delta d \|\v'\|_{L^2}^2)^2}{\langle -L_{\u} \v', \v'\rangle_{L^2}} = \frac{(\partial_t d - \theta \Delta d)^2}{2 \mu}. \]
On the other hand, since $L_{\u}$ is self-adjoint on $L^2(\R)$ and $\psi \perp \v'$ holds, we obtain 
\begin{align*} 
G(\psi_{\min} + \psi) =&\; \frac{\lambda^2}{4} \langle -L_{\u} \v', \v' \rangle_{L^2} + \frac{\lambda}{2} \left(\langle L_{\u} \v', (-L_{\u})^{-1} \psi \rangle_{L^2} - \langle \psi, \v' \rangle_{L^2} \right) + \langle \psi, (-L_{\u})^{-1} \psi \rangle_{L^2} \\
=&\; G(\psi_{\min}) + \langle \psi, (-L_{\u})^{-1} \psi \rangle_{L^2}. 
\end{align*}
Letting $\phi := (-L_{\u})^{-1} \psi$, we see 
\[ \langle \psi, (-L_{\u})^{-1} \psi \rangle_{L^2} = \int_\R 2 \sigma(\u) (\phi')^2 + [B(\u) + D(\u)] \phi^2 \; d\xi \ge 0, \]
which yields 
\[ G(\tilde{\psi}) \ge G(\psi_{\min}) = \frac{(\partial_t d - \theta \Delta d)^2}{2 \mu} \quad \text{for} \; \; \tilde{\psi} \in L^2(\R) \; \; \text{with} \; \; \tilde\psi \perp \v'. \]
Therefore, $\psi_{\min}$ is a minimizer of the minimizing problem \eqref{mini-pro}. 

We here get back to the original minimizing problem. 
Due to the exponential decay estimate \eqref{ex-decay-v}, the function $\bar Q_{\min}$ given by \eqref{min-Q} is well-defined, and it is easily seen that the function is a solution to \eqref{ODE-Q}. 
Therefore, proving the boundedness \eqref{as-Q} replaced $Q$ by $\bar Q_{\min}$ completes the proof of the proposition. 
We next prove the boundedness of 
\[ \partial_\xi \bar Q_{\min}(t,x,\xi) = \frac{1}{(\v')^2(\xi)} \int_{-\infty}^{\xi} \left(\u'(\hat{\xi}) \partial_t d(t,x) - \v'(\hat{\xi}) \Delta d(t,x) - \frac{\lambda(t,x)}{2} L_{\u} \v'(\hat{\xi})\right) \v'(\hat{\xi}) \; d\hat{\xi} \]
on $[0,T] \times \T^d \times \R$. 
Since $\v'(\xi)$ is positive for any $\xi \in \R$, we have 
\begin{equation}\label{bdd-xi-Q} 
\sup_{(t,x,\xi) \in [0,T] \times \T^d \times [-L, L]} |\partial_\xi \bar Q_{\min}(t,x,\xi)| < \infty 
\end{equation}
for any $L>0$. 
Therefore, we discuss the boundedness of $\partial_\xi \bar Q_{\min}(t,x,\xi)$ at far distance for $\xi$. 
Note that $\lim_{\xi \to \pm \infty} \partial_\xi \bar Q_{\min}(t,x,\xi)$ is an undetermined form because of $\lim_{\xi \to \pm \infty} \v'(\xi) = 0$, and we have to remove the undetermined form. 
Due to l'H\^opital's rule, we see 
\[ \lim_{\xi \to \pm \infty} \partial_\xi \bar Q_{\min}(t,x,\xi) = \lim_{\xi \to \pm \infty} \frac{\u'(\xi) \partial_t d(t,x) - \v'(\xi) \Delta d(t,x) - \frac{\lambda(t,x)}{2} L_{\u} \v'(\xi)}{2 \v''(\xi)} \]
if the limit on the right-hand side exists. 
On the other hand, taking the derivative of the identity of $\v'$ as in \eqref{ODE-v'}, we see 
\[\v'' = \dfrac{W'(\u)P'(\u)\u'}{\sqrt{2\tilde{W}(\u)}},\]
where $\tilde{W}$ is the function defined by \eqref{def-tilde-W}. 
Therefore, recalling $\v=P(\u)$ and applying l'H\^opital's rule again, we have
\begin{align*}
\lim_{\xi \to \pm \infty} \frac{\u'(\xi) \partial_t d(t,x) - \v'(\xi) \Delta d(t,x)}{2\v''(\xi)} =&\; \lim_{\xi \to \pm \infty} \frac{\sqrt{2\tilde{W}(\u)}}{2W'(\u)} \left(\frac{\partial_t d}{P'(\u)} - \Delta d \right) \\
=&\; \mp\left( \lim_{u \to \rho_\pm \mp 0} \frac{\tilde{W}(u)}{2 (W'(u))^2}\right)^{1/2} \left(\frac{\partial_t d}{P'(\rho_\pm)} - \Delta d \right) \\
=&\; \mp \left(\frac{P'(\rho_\pm)}{4 W''(\rho_\pm)}\right)^{1/2} \left(\frac{\partial_t d(t,x)}{P'(\rho_\pm)} - \Delta d(t,x) \right). 
\end{align*} 
Differentiating the ordinary differential equation in \eqref{2-3}, we obtain $\v'''=W''(\u) \u'$, which yields 
\begin{align*}
\lim_{\xi \to \pm \infty} \frac{\frac{\lambda(t,x)}{2} L_{\u} \v'(\xi)}{2 \v''(\xi)} =&\; \lim_{\xi \to \pm \infty} \frac{\lambda}{4} \left(\frac{2 \sigma'(\u)\v'' \u' + 2\sigma(\u) \v''' - [B(\u) + D(\u)] \v'}{2\v''}\right) \\
=&\; \lim_{\xi \to \pm \infty} \frac{\lambda}{4} \left(\sigma'(\u) \u' + \frac{\sqrt{2\tilde{W}(\u)}}{2W'(\u)} \left(\frac{2 \sigma(\u) W''(\u)}{P'(\u)} - [B(\u) + D(\u)] \right)\right) \\
=&\; \mp \frac{\lambda(t,x)}{4} \left(\frac{P'(\rho_\pm)}{4 W''(\rho_\pm)}\right)^{1/2} \left(\frac{2 \sigma(\rho_\pm) W''(\rho_\pm)}{P'(\rho_\pm)} - [B(\rho_\pm) + D(\rho_\pm)] \right). 
\end{align*}
Summarizing the above, the derivative $\partial_\xi \bar Q_{\min}(t,x,\xi)$ converges to a finite value as $\xi \to \pm \infty$ depending on $(t,x) \in [0,T] \times \T^d$ and the limiting-values are uniformly bounded with respect to $t$ and $x$. 
Therefore, we conclude that there exists $L>0$ such that $\partial_\xi \bar Q_{\min}(t,x,\xi)$ is bounded on $(t,x,\xi) \in [0,T] \times \T^d \times (\R \setminus (-L, L))$; hence, due to \eqref{bdd-xi-Q}, it holds that 
\[ \sup_{(t,x,\xi) \in [0,T] \times \T^d \times \R} |\partial_\xi \bar Q_{\min}(t,x,\xi)| < \infty. \]
This boundedness also implies that $\bar Q_{\min}$ growth at most linearly in $\xi$, namely, it holds that 
\[ \sup_{(t,x,\xi) \in [0,T] \times \T^d \times \R} \frac{|\bar Q_{\min}(t,x,\xi)|}{1+|\xi|} < \infty. \]
Applying these boundedness estimates to \eqref{ODE-Q}, we also have the boundedness of $\partial_\xi^2 Q_{\min}$. 
Since a similar argument can be applied to the derivatives $\nabla^i \bar Q_{\min}$ for $i=1,2$ and $\partial_t \bar Q_{\min}$, we obtain
\begin{align*} 
\sup_{(t,x,\xi) \in [0,T] \times \T^d \times \R} \Bigg\{&\left(\frac{|\partial_t \bar Q_{\min}|}{1+|\xi|} + |\partial_t \partial_\xi \bar Q_{\min}| + |\partial_t \partial_\xi^2 \bar Q_{\min}|\right) \\
&\; + \sum_{i=0}^2 \left(\frac{|\nabla^i \bar Q_{\min}|}{1+|\xi|} + |\nabla^i \partial_\xi \bar Q_{\min}| + |\nabla^i \partial_\xi^2 \bar Q_{\min}|\right) \Bigg\} < \infty, 
\end{align*}
which yields \eqref{as-Q} replaced $Q$ by $\bar Q_{\min}$. 
\end{proof}

\begin{remark}\label{rem:diff}
(i) General solutions $\bar Q_{\min}^*$ to \eqref{ODE-Q} with bounded derivative $\partial_\xi \bar Q_{\min}^*$ are represented by $\bar Q_{\min}^* = \bar Q_{\min} + C$ with arbitrary constant $C$. 

(ii) In \cite{BBP}, the case $P(u)=u/2$ has been considered. 
In this case, $\bar Q_{\min}$ has a more simple formula. In fact,
we have $\bar Q_{\min} (t,x,\xi)=A(t,x) Q^*(\xi)$, where
\begin{align*}
A(t,x)&=2 \partial_t d(t,x) - \Delta d(t,x),\\
Q^*(\xi)&=\int_0^\xi \dfrac{1}{(\u')^2(\tilde\xi)}
\int_{-\infty}^{\tilde\xi} \left(\u'(\hat{\xi})+\dfrac{\|\u'\|_{L^2}^2 L_{\u}\u'(\hat{\xi})}{\lan-L_{\u}\u', \u'\ran_{L^2}} \right)\u'(\hat\xi) \; d\hat\xi d\tilde\xi.
\end{align*}
This separability of variables for $\bar Q_{\min} (t,x,\xi)$ shows that, roughly speaking, the functional $S_\e (\phi_\e)$ is almost separable into iterated integral along the tangential direction and the normal direction of the surface $\Gamma_t$. 
In our problem, this kind of separation of variables cannot be applied; thus, we have to carefully discuss the dependence of $t$ and $x$ in the singular limit analysis. 
\end{remark}

\subsection{Proof of (1) of Theorem \ref{thm2-1}}\label{sec3-2}
To obtain a lower bound for $S_\e(\phi_\e)$, we start from the variational expression \eqref{2-5}. 
For any test function $H\in C^{1,2}([0,T]\times\T^d)$, we obtain by integrating by parts 
\begin{align*}
S_\e(\phi_\e) \ge J_\e^H(\phi_\e) =&\; \dfrac{1}{\e}\int_0^T \int_{\T^d}
\left\{H \partial_t\phi_\e - H \Delta P(\phi_\e) - \sigma(\phi_\e)|\nabla H|^2 \right\} \; dxdt\\
&\; - \dfrac{1}{\e^3}\int_0^T \int_{\T^d} \left\{ B(\phi_\e)\left(e^{H}-1\right) + D(\phi_\e) \left( e^{-H}-1\right) \right\} \; dx dt.
\end{align*}
We focus on this bound in the case $H(t,x)=\e H_1(t,x,d_\e)$, where $H_1$ is a smooth function satisfying \eqref{as-H}.
Then, it follows from
Proposition \ref{prop3-1} that
each integrand of the above expression can be written as
\begin{align*}
&\; H \partial_t\phi_\e - H \Delta P(\phi_\e) - \sigma(\phi_\e)|\nabla H|^2 \\
=&\; \left[ -\dfrac1\e\v''(d_\e)+ \u'(d_\e)\partial_td - \v'(d_\e)\Delta d - \v'''(d_\e)Q_{d_\e} - 2\v''(d_\e)\partial_\xi Q_{d_\e} -\v'(d_\e)\partial_\xi^2 Q_{d_\e} \right] H_{1, d_\e}\\
&\; -\sigma(\u(d_\e)) (\partial_\xi H_{1,d_\e})^2 
+ \tilde{\mc R}_{\e,d_\e}^{(1)} R_\e + \e \tilde{\mc R}_{\e,d_\e}^{(2)}
\end{align*}
and
\begin{align*}
&\; B(\phi_\e)\left(e^{H}-1\right) + D(\phi_\e) \left( e^{-H}-1\right) \\
=&\; \e\left[B(\u(d_\e)) - D(\u(d_\e))\right]H_{1,d_\e} + \e^2\left[B'(\u(d_\e)) - D'(\u(d_\e))\right]
\u'(d_\e)Q_{d_\e} H_{1, d_\e}\\
& +\frac{\e^2}{2}\left[B(\u(d_\e)) + D(\u(d_\e))\right]H_{1, d_\e}^2
+\e^2 \tilde{\mc R}_{\e,d_\e}^{(3)} R_\e + \e^3 \tilde{\mc R}_{\e, d_\e}^{(4)}, 
\end{align*}
where $\tilde{\mc R}_{\e, d_\e}^{(i)}$ represents $\tilde{\mc R}_\e^{(i)}(t,x,d_\e)$ with some function $\tilde{\mc R}^{(i)}_\e: [0,T] \times \T^d \times \R \to \R, i=1,2,3,4,$ satisfying 
\[ \varlimsup_{\e \to 0} \sup_{(t, x, \xi) \in [0,T] \times \T^d \times \R} e^{\gamma' |\xi|} |\tilde{\mc R}_\e^{(i)} (t, x, \xi)| < \infty \]
for some $\gamma' > 0$. 
Therefore, using \eqref{2-3} and its differential 
\begin{align*}
\v'''+ \left( B'(\u) -D'(\u) \right)\u'=0,
\end{align*}
we have 
\begin{align*}
&\; J_\e^H(\phi_\e) \\
=&\; \dfrac{1}{\e}\int_0^T \int_{\T^d} \left[\u'(d_\e)\partial_td - \v'(d_\e)\Delta d - 2\v''(d_\e)\partial_\xi Q_{d_\e} -\v'(d_\e)\partial_\xi^2 Q_{d_\e} \right] H_{1, d_\e} \\
&\; \quad -\sigma(\u(d_\e)) (\partial_\xi H_{1,d_\e})^2 - \frac{\e^2}{2}\left[B(\u(d_\e)) + D(\u(d_\e))\right]H_{1, d_\e}^2 + (\tilde{\mc R}_{\e, d_\e}^{(1)} + \tilde{\mc R}_{\e, d_\e}^{(3)}) R_\e \; dx dt \\
&\; + \int_0^T \int_{\T^d} \tilde{\mc R}_{\e, d_\e}^{(2)} + \tilde{\mc R}_{\e, d_\e}^{(4)} \; dxdt. 
\end{align*}
Due to the convergence assumption of $R_\e$ as in \eqref{as-R}, Proposition \ref{prop:con-e-int} yields 
\begin{align*}
\varliminf_{\e\to0} S_\e(\phi_\e) \ge \int_0^T \int_{\Gamma_t} \int_\R
\left\{ F_QH_1-\sigma(\u)(\partial_\xi H_1)^2 -\dfrac12\left[B(\u)+D(\u)\right]H_1^2\right\} \; d\xi d \mathcal{H}^{d-1} dt,
\end{align*}
where $F_Q$ is the function defined by \eqref{3-2}. 
Notice that the integral in $\xi$ is equal to
\begin{align*}
\lan F_Q+\dfrac{1}{2}\L H_1,H_1\ran_{L^2}.
\end{align*}
Since Proposition \ref{prop:Lu} yields that $(-L_{\u})^{-1}F_Q$ satisfies \eqref{as-H} as the estimate of $\hat{H}_1$ in the proof of Proposition \ref{prop:dec-H}, 
choosing $H_1 = (-\L)^{-1}F_Q$ and using Proposition \ref{prop3-2} and Proposition \ref{prop:dist},
we obtain
\begin{align*}
\varliminf_{\e\to0} S_\e(\phi_\e)
&\ge \dfrac12\int_0^T \int_{\Gamma_t} \int_\R F_Q(-\L)^{-1}F_Q \; d\xi d\mathcal{H}^{d-1} dt\\
&\ge \int_0^T \int_{\Gamma_t} \dfrac{(\partial_td-\theta\Delta d_t)^2}{4\mu}\; d\mathcal{H}^{d-1} dt\\
&= \int_0^T \int_{\Gamma_t} \dfrac{(v_t-\theta\kappa_t)^2}{4\mu} \; d\mathcal{H}^{d-1} dt =S_\ac(\Gamma),
\end{align*}
which completes the proof of (1) of Theorem \ref{thm2-1}.

\subsection{Proof of (2) of Theorem \ref{thm2-1}}\label{sec3-3}

We prove (2) of Theorem \ref{thm2-1} with the choice $\hat{Q} = \bar Q_{\min}$, where $\bar Q_{\min}$ is the minimizer obtained in Proposition \ref{3-2}. 
Recall the formula for $S_\e(\phi_\e)$ from \eqref{2-5}:
\begin{align*}
S_\e(\phi_\e)&=\dfrac{1}{\e} \int_0^T \int_{\T^d} \sigma(\phi_\e)|\nabla H_{\max, \e}|^2 \; dxdt\notag\\
& \quad + \dfrac{1}{\e^3} \int_0^T \int_{\T^d} B(\phi_\e)\left(1-e^{H_{\max, \e}}+H_{\max, \e}e^{H_{\max, \e}} \right) \; dxdt \notag\\
& \qquad + \dfrac{1}{\e^3} \int_0^T \int_{\T^d} D(\phi_\e) \left(1-e^{-H_{\max, \e}}-H_{\max, \e}e^{-H_{\max, \e}}\right) \; dxdt\\
& =: I_1+I_2+I_3.
\end{align*}
Due to Proposition \ref{prop:dec-H}, we can decompose $H_{\max, \e}$ as 
\[ H_{\max, \e}(t,x) = \e \hat{H}_1(t,x,d_\e) + \e^2 \hat{K}_\e(t,x) \quad \text{for} \; \; (t,x) \in [0,T] \times \T^d, \]
where $\hat{H}_1$ is the function defined by \eqref{linear-H} replaced $Q$ by $\bar Q_{\min}$.  
Due to the estimates \eqref{es-H*} and \eqref{es-He},
Proposition \ref{prop3-1} can be applied to calculate the integrant of $I_1$ as
\begin{align*}
&\; \sigma(\phi_\e)|\nabla H_{\max, \e}|^2 \\
=&\; [\sigma(\u(d_\e)) + \e \mc R_\e^{(4)}(t,x,d_\e)] \cdot [(\partial_\xi \hat{H}_{1, d_\e})^2 |\nabla d|^2 + \e \hat{\mc R}_\e^{(1)}(t,x,d_\e) + O_{\hat{K}_\e}(\e^4)] \\
=&\; \sigma(\u(d_\e))(\partial_\xi \hat{H}_{1, d_\e})^2 + \e \hat{\mc R}_\e^{(2)} (t,x,d_\e) + O_{\hat{K}_\e}(\e^4), 
\end{align*}
where $\hat{\mc R}_\e^{(i)}: [0,T] \times \T^d \times \R \to \R, i=1,2,$ represents a function satisfying 
\begin{equation}\label{reminders-ex-decay}
\varlimsup_{\e \to 0} \sup_{(t, x, \xi) \in [0,T] \times \T^d \times \R} e^{\gamma' |\xi|} |\hat{\mc R}_\e^{(i)} (t, x, \xi)| < \infty
\end{equation}
for some $\gamma' > 0$. 
Here, we have used the estimate \eqref{pr-d} to obtain the second equality.
We thus obtain 
\[ I_1 = \frac{1}{\e} \int_0^T \int_{\T^d}  \sigma(\u(d_\e))(\partial_\xi \hat{H}_{1, d_\e})^2 \; dxdt + \int_0^T \int_{\T^d} \hat{\mc R}_\e^{(2)} (t,x,d_\e) + O_{\hat{K}_\e}(\e^3) \; dxdt \]
Similarly,
Proposition \ref{prop3-1} yields 
\begin{align*}
I_2=&\; \frac{1}{\e} \int_0^T \int_{\T^d} \frac{B(\u(d_\e))}{2} (\hat{H}_{1, d_\e})^2 \; dxdt + \int_0^T \int_{\T^d} \hat{\mc R}_\e^{(3)} (t,x,d_\e) + O_{\hat{K}_\e}(\e^3) \; dxdt, \\
I_3=&\; \frac{1}{\e}\int_0^T \int_{\T^d} \frac{D(\u(d_\e))}{2} (\hat{H}_{1, d_\e})^2 \; dxdy + \int_0^T \int_{\T^d} \hat{\mc R}_\e^{(4)} (t,x,d_\e) + O_{\hat{K}_\e}(\e^3) \; dxdt,
\end{align*}
where $\hat{\mc R}^{(3)}$ and $\hat{\mc R}^{(4)}$ are functions satisfying \eqref{reminders-ex-decay}.  
Therefore, using Proposition \ref{prop:con-e-int}, we have 
\begin{align*}
\lim_{\e \to 0} S_\e(\phi_\e) =&\; \int_0^T \int_{\Gamma_t} \int_\R \sigma(\u) (\partial_\xi \hat{H}_1)^2 + \frac{B(\u) + D(\u)}{2} (\hat{H}_1)^2 \; d\xi d\mathcal{H}^{d-1} dt \\
=&\; \frac{1}{2} \int_0^T \int_{\Gamma_t} \int_\R \left\{-2\partial_\xi[ \sigma(\u) \partial_\xi \hat{H}_1] +[B(\u) + D(\u)] \hat{H}_1 \right\} \hat{H}_1 \; d\xi d\mathcal{H}^{d-1} dt \\
=&\; \frac{1}{2} \int_0^T \int_{\Gamma_t} \int_\R (-L_{\u} \hat{H}_1) \hat{H}_1 \; d\xi d\mathcal{H}^{d-1} dt. 
\end{align*}
Notice, comparing \eqref{linear-H} with the definition of $F_Q$ as in \eqref{3-2}, we have $-L_{\u} \hat{H}_1 = F_{\bar Q_{\min}}$. 
We thus see by Proposition \ref{prop3-2} and Proposition \ref{prop:dist} 
\begin{align*}
\lim_{\e \to 0} S_\e(\phi_\e) =&\; \frac{1}{2} \int_0^T \int_{\Gamma_t} \int_\R F_{\bar Q_{\min}} (-L_{\u})^{-1} F_{\bar Q_{\min}} \; d\xi d\mathcal{H}^{d-1} dt \\
=&\; \int_0^T \int_{\Gamma_t} \frac{(\partial_t d - \theta \Delta d)^2}{4 \mu} \; d\mathcal{H}^{d-1} dt = S_{\rm as}(\Gamma), 
\end{align*}
which completes the proof of (2) of Theorem \ref{thm2-1}. 

\appendix

\section{Analytic properties}\label{seca}

In this section, we summarize some properties of functions and operators related to our problem. 
We first derive an identity for $\v'$ and calculate inner products related to $\u'$ and $\v'$ as follows. 

\begin{proposition}
Let $\u: \R \to \R$ and $\v: \R \to \R$ be the solutions of \eqref{pro-translayer} and \eqref{2-3}, respectively. 
Define $\tilde{W}: \R \to \R$ by \eqref{def-tilde-W}. 
Then, it holds that 
\begin{equation}\label{ODE-v'}
\v'(\xi) = \sqrt{2 \tilde{W}(\u(\xi))} \quad \text{for} \; \; \xi \in \R
\end{equation}
and 
\begin{equation}\label{inner-uv}
\langle \u', \v' \rangle_{L^2} = \int_{\rho_-}^{\rho_+} \sqrt{2 \tilde{W}(\rho)} \; d\rho, \quad \langle \v', \v' \rangle_{L^2} = \int_{\rho_-}^{\rho_+} P(\rho)\sqrt{2\tilde{W}(\rho)} \; d\rho. 
\end{equation}
\end{proposition}

\begin{proof}
Since $\u$ and $\v$ satisfy 
\[ \v''(\xi) = W'(\u(\xi)) \quad \text{for} \; \; \xi \in \R, \quad \lim_{\xi \to -\infty} \v'(\xi) = 0, \quad \lim_{\xi \to -\infty} \u(\xi) = \rho_-, \]
multiplying $\v'$ by its ordinary differential equation and integrating it, we have by $\v = P(\u)$ 
\begin{align*} 
\frac{1}{2} (\v'(\xi))^2 =&\; \int_{-\infty}^{\xi} W'(\u(\tilde{\xi})) \v'(\tilde{\xi}) \; d\tilde{\xi} =\int_{- \infty}^{\xi} W'(u(\tilde{\xi})) P'(u(\tilde{\xi})) \u'(\tilde{\xi}) \; d\tilde{\xi} \\
=&\; \int_{\rho_-}^{\u(\xi)} W'(\rho) P'(\rho) \; d\rho = \tilde{W}(\u(\xi)) \quad \text{for} \; \; \xi \in \R. 
\end{align*}
Due to $\v' > 0$, the identity above shows \eqref{ODE-v'}. 
Furthermore, due to $\u(\xi) \to \rho_\pm$ as $\xi \to \pm \infty$, \eqref{ODE-v'} yields 
\[ \langle \u', \v' \rangle_{L^2} = \int_\R \u'(\xi) \sqrt{2 \tilde{W}(\u(\xi))} \; d\xi = \int_{\rho_-}^{\rho_+} \sqrt{2 \tilde{W}(\rho)} \; d\rho \]
and 
\[ \langle \v', \v' \rangle_{L^2} = \int_\R P'(\u(\xi)) \u'(\xi) \sqrt{2 \tilde{W}(\u(\xi))} \; d\xi = \int_{\rho_-}^{\rho_+} P(\rho)\sqrt{2\tilde{W}(\rho)} \; d\rho, \]
which complete the calculation of \eqref{inner-uv}. 
\end{proof}

We next discuss properties of the linear operator $L_{\u}$ introduced in \eqref{def-L}. 
Let $\mc H_k=\mc H_k(\R)$ be the Sobolev space of order $k=1,2$ on $\R$:
\begin{align*}
\mc H_1=\left\{\psi\in L^2(\R): \int_\R \psi(\xi)^2+\psi'(\xi)^2 d\xi <\infty \right\}
\quad \text{and} \quad \mc H_2=\left\{\psi\in \mc H_1: \psi'\in \mc H_1 \right\}.
\end{align*}
Let $D(L_{\u})$ be the set of all function $\psi$ such that
there exists an $L^2(\R)$-function $\phi\in \mc H_1$ satisfying
\begin{align*}
-\lan 2\sigma(\u)\psi', \varphi' \ran_{L^2} - \lan [B(\u)+D(\u)]\psi, \varphi \ran_{L^2} = \lan \phi, \varphi \ran_{L^2},
\end{align*}
for any $\varphi\in \mc H_1$.
Note that such $\phi\in L^2(\R)$ is unique for any given $\psi\in D(L_{\u})$.
Therefore the application $\psi\in D(L_{\u})\mapsto \phi\in L^2(\R)$ is well-defined.
We denote this linear operator by $L_{\u}$. Moreover, one can easily see that $D(\L)=\mc H_2$.

The linear operator $(L_{\u}, D(L_{\u}))$ enjoys the following properties.

\begin{proposition}\label{prop:ele}
The linear operator $(L_{\u}, D(L_{\u}))$ is bijective, self-adjoint on $L^2(\R)$
and has a bounded inverse.
\end{proposition}
\begin{proof}
The proof is elementary but we give a complete proof for the reader's convenience.

Let $\psi\in D(L_{\u})$ and note that
\begin{align*}
\lan -L_{\u}\psi, \psi \ran_{L^2}
= \lan 2\sigma(\u)\psi', \psi'\ran_{L^2} + \lan [B(\u)+D(\u)]\psi, \psi\ran_{L^2}.
\end{align*}
Since $B(\u)+D(\u)$ is bounded away from $0$, there exists a constant $C_0>0$ such that
\begin{align}\label{ele-ine}
\|\psi\|_{L^2}^2 \le C_0 \lan -L_{\u}\psi, \psi \ran_{L^2}.
\end{align}
Therefore if $L_{\u}\psi=0$, then $\psi=0$, which implies that $L_{\u}$ is injective.
To show that $L_{\u}$ is surjective,
we define the Hilbert space $\mc H_{\L}$,
which consists of all functions $\psi\in \mc H_1$ and its inner product $\lan\cdot,\cdot\ran_{\mc H_{\L}}$ is given by
\begin{align*}
\lan \psi, \phi\ran_{\mc H_{\L}} = \lan 2\sigma(\u)\psi', \phi' \ran_{L^2} + \lan [B(\u)+D(\u)]\psi, \phi \ran_{L^2}, \quad \psi,\phi\in \mc H_{\L}.
\end{align*}
Note that the corresponding norm is equivalent to the $\mc H_1$-Sobolev norm $\|\cdot\|_{\mc H_1}$. 
Fix $w\in L^2(\R)$. Define the functional $I_w$ on $\mc H_{\L}$ by
$I_w(\varphi)=-\lan \varphi, w\ran_{L^2}, \varphi\in \mc H_{\L}$.
Since there exists a constant $C_0>0$ such that
\[|I_w(\varphi)|\le\| w\|_{L^2}\|\varphi\|_{L^2}\le C_0 \| w\|_{L^2}\|\varphi\|_{\mc H_{\L}},\]
for any $\varphi\in \mc H_{\L}$,
$I_w$ is continuous on $\mc H_{\L}$. By Risez's representation theorem,
there exists $\psi\in \mc H_{\L}$ such that for any $\varphi\in\mc H_{\L}$ it holds that
\begin{align*}
\lan \psi, \varphi \ran_{\mc H_{\L}}= I_w(\varphi) = -\lan \varphi, w\ran_{L^2}.
\end{align*}
This implies that $\psi\in D(\L)$ and $\L\psi=w$.
Therefore $L_{\u}:D(\L)\to L^2(\R)$ is surjective.

We next show that $(L_{\u}, D(L_{\u}))$ is self-adjoint on $L^2(\R)$.
Let $\L^*$ be the adjoint operator of $\L$  on $L^2(\R)$ and $D(\L)$ its domain.
Since $D(\L)=\mc H_2$, an integration by parts shows
$\lan\L\psi, \phi\ran_{L^2}=\lan \psi, \L\phi\ran_{L^2}$
for any $\psi, \phi\in D(\L)$. Therefore $\L\subset \L^*$.
Conversely, take any $\varphi\in D(\L^*)$.
By the definition of the adjoint operator $\L^*$,
it holds that for any $\psi\in D(\L)$
\begin{align}\label{ele-ine1}
\lan \L\psi, \varphi\ran_{L^2}=\lan\psi, \L^*\varphi\ran_{L^2}.
\end{align}
On the other hand, letting $w=\L^*\varphi$ in the previous paragraph,
there exists $\phi\in \mc H_{\L}$ such that for any $\psi\in\mc H_{\L}$ it holds that
\begin{align}\label{ele-ine2}
\lan \phi, \psi \ran_{\mc H_{\L}}= I_{\L^*\varphi}(\psi) = -\lan \psi, \L^*\varphi\ran_{L^2}.
\end{align}
This implies that $\phi\in D(\L)$ and $\L\phi=\L^*\varphi$.
It follows from the definition of $\lan\cdot,\cdot\ran_{\mc H_{\L}}$, \eqref{ele-ine1} and \eqref{ele-ine2} that for any $\psi\in D(\L)$
\begin{align*}
\lan\L\psi, \phi\ran_{L^2}=-\lan\psi,\phi\ran_{\mc H_{\L}}
=\lan \psi, \L^*\varphi\ran_{L^2}=\lan \L\psi, \varphi\ran_{L^2}.
\end{align*}
Since $\L$ is surjective, we have $\phi=\varphi$. Therefore $\L^*\subset\L$ and
$\L$ is self-adjoint on $L^2(\R)$.

Since $\L$ is a bijection from $D(\L)$ to $L^2(\R)$, $\L$ has an inverse operator. 
The boundedness of the inverse operator immediately follows from \eqref{ele-ine},
which completes the proof of Proposition \ref{prop:ele}.
\end{proof}

We further present the following estimate. 

\begin{proposition}\label{prop:Lu}
There exist constants $C> 0$ and $\gamma_0 > 0$ such that if $w: \R \to \R$ and $\gamma'>0$ satisfies 
\[ \sup_{\xi \in \R} |w(\xi)| e^{\gamma'|\xi|} < \infty, \]
then $\psi = \L^{-1} w$ satisfies 
\begin{equation}\label{Lu-ine} 
\sup_{\xi \in \R} \{|\psi(\xi)| + |\psi'(\xi)| + |\psi''(\xi)|\} e^{\tilde{\gamma}'|\xi|}
\le C \sup_{\xi \in \R} |w(\xi)| e^{\gamma'|\xi|}, 
\end{equation}
where $\tilde{\gamma}' = \min\{\gamma_0/2, \gamma'\}$. 
\end{proposition}

\begin{proof}
Fix a smooth function $a:\R\to[0,\infty)$ whose derivatives are bounded
and $w\in L^2(\R)$ and set $\psi = \L^{-1} w$.
We start by observing
\begin{align*}
\lan -aw, \psi\ran_{L^2}
&= \lan -aL_{\u}\psi, \psi \ran_{L^2} \\
&= \lan 2\sigma(\u)\psi', (a\psi)'\ran_{L^2} + \lan a[B(\u)+D(\u)]\psi, \psi\ran_{L^2}\\
&= \lan 2\sigma(\u)\psi', a'\psi\ran_{L^2} + \lan 2\sigma(\u)\psi', a\psi'\ran_{L^2} + \lan a[B(\u)+D(\u)]\psi, \psi\ran_{L^2}.
\end{align*}
Since $2\psi'\psi=(\psi^2)'$, the integration by parts on the first term shows
\begin{align*}
\lan 2\sigma(\u)\psi', a'\psi\ran_{L^2} = -\lan(\sigma(\u)a')', \psi^2 \ran_{L^2} = -\lan\sigma'(\u)\u'a', \psi^2 \ran_{L^2} -\lan\sigma(\u)a'', \psi^2 \ran_{L^2}.
\end{align*}
Therefore, since $\sigma, B,D$ are away from $0$, we have 
\begin{align*}
\lan -aw, \psi\ran_{L^2} \ge \lan c_0a - \sigma'(\u)\u'a' - \sigma(\u)a'', \psi^2 \ran_{L^2} + c_0\lan a\psi', \psi'\ran_{L^2},
\end{align*}
for some $c_0>0$.
On the other hand, due to Young's inequality, it holds that 
\begin{align*}
\lan -aw, \psi\ran_{L^2}\le \dfrac{1}{2c_0}\lan aw,w\ran_{L^2} + \dfrac{c_0}{2}\lan a\psi, \psi \ran_{L^2}.
\end{align*}
Consequently, we have shown that 
\begin{equation}\label{Lu-ine1}
\lan a - 2c_0^{-1}[\sigma'(\u)\u'a' + \sigma(\u)a''], \psi^2 \ran_{L^2} + 2 \lan a\psi', \psi'\ran_{L^2}
\le c_0^{-2}\lan aw,w\ran_{L^2}.
\end{equation}

For each $\gamma\in\R$, define $a_\gamma(\xi)=e^{\gamma\xi}$.
We construct smooth functions $\{b_n:\R\to[0,\infty)\}_{n\in\N}$ satisfying
\begin{itemize}
\item[(i)] for each $n\in\N$, all derivatives of $b_n$ are bounded;
\item[(ii)] for any $\xi\in\R$ and $n\in\N$, $b_n(\xi)\le a_\gamma(\xi)$;
\item[(iii)] there exists $\gamma_0>0$ such that for any $n\in\N$ and $0<\gamma\le \gamma_0$, it holds that
\begin{align*}
b_n - 2c_0^{-1}[\sigma'(\u)\u'b_n' + \sigma(\u)b_n'']\ge 0;
\end{align*}
\item[(iv)] for any $\xi\in\R$ and $k=0,1,2$, $\lim_{n\to\infty}b_n^{(k)}(\xi)=a_\gamma^{(k)}(\xi)$.
\end{itemize}
For this purpose, fix a smooth function $f:[0,1]\to[0,1]$ satisfying
$f(0)=1, f(1)=0$ and $f^{(k)}(0)=f^{(k)}(1)=0$ for any $k\in\N$.
For each $n\in\N$, let $\delta_n:\R\to[0,\gamma]$ be the smooth function defined by
\begin{align*}
\delta_n(\xi)=
\begin{cases}
\gamma, \quad \xi\le n,\\
\gamma f\left((\xi-n)/n\right), \quad n \le \xi \le 2n,\\
0, \quad \xi \ge 2n,
\end{cases}
\end{align*}
and define $b_n(\xi)=e^{\delta_n(\xi)\xi}$.
Then it is clear that (i), (ii) and (iv) hold. It remains to check (iii).

Note that $b_n$ coincides with $a_\gamma$ on $(-\infty,n]$ and
\begin{align*}
a_\gamma - 2c_0^{-1}[\sigma'(\u)\u'a_\gamma' + \sigma(\u)a_\gamma''] = \left\{ 1 - 2c_0^{-1}[ \sigma'(\u)\u'\gamma + \sigma(\u)\gamma^2]\right\}a_\gamma.
\end{align*}
Since $\sigma, \sigma', \u'$ are bounded, if we take a sufficiently small $\gamma_0>0$,
for any $0<\gamma\le\gamma_0$ the right-hand side is larger than or equal to $(1/2)a_\gamma$.
Hence for any $n\in\N$ and $0<\gamma\le \gamma_0$ (iii) holds on $(-\infty,n]$.
It is clear that (iii) holds on $[2n,\infty)$. By elementary calculations, we have
\begin{align*}
&b_n'(\xi)=(\delta_n'(\xi)\xi+\delta_n(\xi))b_n(\xi),\\
&b_n''(\xi)=\left\{ \delta_n''(\xi)\xi+2\delta_n'(\xi) + (\delta_n'(\xi)\xi+\delta_n(\xi))^2 \right\} b_n(\xi),\\
&\sup_{n \le \xi \le 2n} |\delta_n(\xi)| = \gamma, \quad \sup_{n\le \xi \le 2n}\left|\delta_n'(\xi)\right|=\dfrac{\gamma}{n}\|f'\|_\infty, \quad
\sup_{n\le \xi \le 2n}\left|\delta_n''(\xi)\right|=\dfrac{\gamma}{n^2}\|f''\|_\infty.
\end{align*}
Therefore, re-choosing $\gamma_0$ small if necessary, we have 
\[ 2c_0^{-1}[\sigma'(\u)\u'b_n' + \sigma(\u)b_n''] \le b_n \]
for any $n \in \mathbb{N}$ and $0 < \gamma \le \gamma_0$, which yields (iii) on $[n,2n]$.

Define $\tilde{\gamma}' := \min\{\gamma_0/2, \gamma'\}$ and let $\{b_n:\R\to[0,\infty)\}_{n\in\N}$ be smooth functions satisfying properties (i)--(iv) with $\gamma = 2\tilde{\gamma}'$. 
By (i), we can substitute $a = b_n$ into \eqref{Lu-ine1} to obtain 
\begin{align*}
\lan b_n - 2c_0^{-1}[\sigma'(\u)\u'b_n' - \sigma(\u)b_n''], \psi^2 \ran_{L^2} + \lan b_n\psi', \psi'\ran_{L^2}
\le c_0^{-2}\lan b_nw,w\ran_{L^2}.
\end{align*}
By (ii), the right-hand side is bounded above from $c_0^{-2}\lan a_{2\tilde{\gamma}'} w,w\ran_{L^2}$.
On the other hand, by (iii), (iv) and Fatou's lemma, we have
\begin{align*}
\lan a_{2\tilde{\gamma}'} - 2c_0^{-1}[\sigma'(\u)\u'a_{2\tilde{\gamma}'}' - \sigma(\u)a_{2\tilde{\gamma}'}''], \psi^2 \ran_{L^2} + 2 \lan a_{2\tilde{\gamma}'}\psi', \psi'\ran_{L^2}
\le c_0^{-2} \lan a_{2\tilde{\gamma}'} w,w\ran_{L^2}. 
\end{align*}
As we have seen, for any $0<\gamma\le\gamma_0$, it holds that 
\begin{align*}
a_\gamma - 2c_0^{-1}[\sigma'(\u)\u'a_\gamma' - \sigma(\u)a_\gamma''] \ge (1/2)a_\gamma.
\end{align*}
Collecting the previous estimates, we have 
for any $0<\gamma\le\gamma_0$
\begin{align*}
\lan a_{2\tilde{\gamma}'}\psi, \psi \ran_{L^2} + \lan a_{2\tilde{\gamma}'}\psi', \psi'\ran_{L^2}
\le 2c_0^{-2}\lan a_{2\tilde{\gamma}'}w,w\ran_{L^2}.
\end{align*}
We thus obtain by elementary calculations 
\begin{equation}\label{Lu-ine2} 
\|\psi(\xi)e^{\tilde{\gamma}' \xi} \|_{\mc H_1} \le C_0 \| w(\xi)e^{\gamma'\xi}\|_{L^2}, 
\end{equation}
where $C_0$ is a positive constant. 
Notice that $\psi''$ can be written as a linear combination of $\psi, \psi'$ and $w$ with bounded coefficients, since $\psi$ satisfies $L_{\overline{u}} \psi = w$. 
Therefore, it holds that 
\[ \|(\psi(\xi)e^{\tilde{\gamma}' \xi})''\|_{L^2} \le C_1 (\|\psi(\xi)e^{\tilde{\gamma}' \xi}\|_{\mc H_1} + \|w(\xi) e^{\tilde{\gamma}' \xi}\|_{L^2}) \]
for some constant $C_1 > 0$, which yields by combining \eqref{Lu-ine2} and the Sobolev inequality 
\[ \| \psi(\xi)e^{\tilde{\gamma}' \xi}\|_{\mc H_1} \le C_2 C_1(C_0 + 1) \| w(\xi)e^{\gamma'|\xi|}\|_{L^2} \]
if $C_2>0$ is defined as the Sobolev constant. 
Here, we have used $\tilde{\gamma}' \le \gamma'$. 
Applying $L_{\overline{u}} \psi = w$ again, we obtain 
\[ \sup_{\xi \in \R}\{|\psi (\xi)| + |\psi'(\xi)| + |\psi''(\xi)|\}e^{\tilde{\gamma}' \xi} \le \tilde{C} \|\psi(\xi)e^{\tilde{\gamma}'\xi}\|_{C^2} \le C \| w(\xi)e^{\gamma'|\xi|}\|_{L^2} \]
for some constants $C, \tilde{C}>0$. 
Performing the same argument for $a_{-2\tilde{\gamma}'}(\xi)=e^{-2\tilde{\gamma}'\xi}$, we can strengthen the previous estimate to \eqref{Lu-ine}.
\end{proof}

\section{Geometric properties}\label{secb}

In this section, we summarize geometric properties related to our problem. 
The following proposition lists derivatives of the distance function from surfaces. 

\begin{proposition}\label{prop:dist}
Let $\Gamma = \{\Gamma_t\}_{t \in [0,T]}$ be a family of oriented smooth hyper-surfaces with $\Gamma_t = \partial \Omega_t$ for some open $\Omega_t \subset \T^d$. 
Assume $\Gamma_t$ has a finite surface area for any $t \in [0,T]$. 
Denote by $n_t$ the outward unit normal vector of $\Gamma_t$, by $v_t$ the normal velocity of $\Gamma_t$, and by $h_t$ the mean curvature of $\Gamma_t$. 
Let $\kappa$ be the constant defined around \eqref{def-K}. 
Denote by $d(x,t)$ be a regularized version of the signed distance from $\Gamma_t$ satisfying \eqref{def-dist}. 
Then, it holds that 
\begin{align*}
&\nabla d(t,x) = n_t (\zeta(t, x)), \\
&\Delta d(t,x) = \sum_{k=1}^{d-1} \frac{-\kappa_{t,k}(\zeta(t,x))}{1 - \kappa_{t,k}(\zeta(t,x))d(t,x)}, \\
&\partial_t d(t,x) = -v_t(\zeta(t,x)), 
\end{align*}
for $(t,x)$ such that $|{\rm dist}(x, \Gamma_t)| \le \kappa$, where $\zeta(t,x) \in \Gamma_t$ be the unique point satisfying $|{\rm dist}(x, \Gamma_t)| = |x- \zeta(t,x)|$ and $\kappa_{t, 1}(\zeta(t,x)), \kappa_{t, 2}(\zeta(t,x)), \ldots, \kappa_{t, d-1}(\zeta(t,x))$ are the principal curvatures of $\Gamma_t$ at $\zeta(t,x)$ such that the curvatures are positive if $\Gamma_t$ is convex. 
\end{proposition}

\begin{proof}
We refer to \cite[Section 14.6]{GT} for the equalities of $\nabla d$ and $\Delta d$. 
The smoothness of $d(t,x)$ and $\zeta(t,x)$ also can be proved by using a similar choice of the coordinate of $\T^d$ to it in \cite[Section 14.6]{GT}; thus we omit the proof of the smoothness. 
By taking the time derivative of $|x - \zeta(t,x)|^2 = d(t,x)^2$, we have 
\[ 2 \langle x- \zeta(t,x), - \partial_t \zeta (t,x) \rangle_{\T^d} = 2 d(t,x) \partial_t d(t, x). \]
Noticing $n_t = (x- \zeta(t,x))/d(t,x)$, it holds that 
\[ \partial_t d(t,x) = \left \langle \frac{x-\zeta(t,x)}{d(t,x)}, - \partial_t \zeta (t,x) \right \rangle_{\T^d} = -v_t(t,x), \]
which completes the proof. 
\end{proof}

In order to calculate the singular limit of functionals, we repeatedly use the following formulas. 

\begin{proposition}\label{prop:con-e-int}
Let $\Gamma = \{\Gamma_t\}_{t \in [0,T]}$ be a family of oriented smooth hyper-surfaces with $\Gamma_t = \partial \Omega_t$ for some open $\Omega_t \subset \T^d$. 
Assume $\Gamma_t$ has a finite surface area for any $t \in [0,T]$. 
Denote by $d(x,t)$ be a regularized version of the signed distance from $\Gamma_t$ satisfying \eqref{def-dist}. 
Let $\gamma'>0$ be an arbitrary positive constant. 
Then, the following statements hold: 
\begin{itemize}
\item[(1)] Let $\tilde{\mc R}_\e: [0,T] \times \T^d \times \R \to \R$ be a continuous function satisfying 
\begin{equation}\label{as-decay-tR} 
\varlimsup_{\e \to 0} \sup_{(t, x, \xi) \in [0,T] \times \T^d \times \R} e^{\gamma' |\xi|} |\tilde{\mc R}_\e(t, x, \xi)| =0 
\end{equation} 
Then, it holds that 
\begin{equation}\label{con-decay-tR1} 
\lim_{\e \to 0} \frac{1}{\e} \int_0^T \int_{\T^d} \tilde{\mc R}_\e (t,x, d(t,x)/\e) \; dxdt = 0. 
\end{equation}

\item[(2)] Let $A: [0,T] \times \T^d \times \R \to \R$ be a continuous function satisfying 
\[ \sup_{(t, x, \xi) \in [0,T] \times \T^d \times \R} e^{\gamma' |\xi|} |A(t, x, \xi)| < \infty \]
Then, it holds that 
\begin{equation}\label{con-decay-tR2} 
\lim_{\e \to 0} \frac{1}{\e} \int_0^T \int_{\T^d} A(t,x, d(t,x)/ \e) \; dxdt = \int_0^T \int_{\Gamma_t} \int_\R A(t,x, \xi) \; d\xi d\mathcal{H}^{d-1}(dx) dt. 
\end{equation}
\end{itemize}
\end{proposition}

\begin{proof}
Let $\kappa$ be the constant defined around \eqref{def-K}. 
We first prove (1). 
By means of \eqref{as-decay-tR}, letting
\[ C_\e := \sup_{(t, x, \xi) \in [0,T] \times \T^d \times \R} e^{\gamma' |\xi|} |\tilde{\mc R}_\e(t, x, \xi)|, \]
we see that $C_\e \to 0$ as $\e \to 0$. 
We decompose the integral as 
\begin{align*}
&\; \frac{1}{\e} \int_0^T \int_{\T^d} \tilde{\mc R}_\e (t,x, d_\e) \; dxdt \\
=&\; \frac{1}{\e} \int_0^T \int_{\{x: |d(t,x)| \le \kappa \}} \tilde{\mc R}_\e (t,x, d_\e) \; dxdt + \frac{1}{\e} \int_0^T \int_{\{x: |d(t,x)| > \kappa\}} \tilde{\mc R}_\e (t,x, d_\e) \; dxdt =: I_1 + I_2,  
\end{align*}
where $d_\e = d(t,x)/ \e$. 
Due to \eqref{as-decay-tR}, we have by a simple calculation 
\begin{equation}\label{decay-tR1}
|I_2| \le \frac{T C_\e}{\e} e^{-\gamma' \kappa/ \e}, 
\end{equation}
which implies that $I_2 \to 0$ as $\e \to 0$. 
Note that Proposition \ref{prop:dist} yields $|\nabla d| = 1$ if $(t,x)$ satisfies $|d(t,x)| \le \kappa$. 
We thus apply the co-area formula (see \cite[Theorem 3.10]{EG} for example) to $I_1$ to obtain 
\begin{align*}
|I_1| \le&\; \frac{1}{\e} \int_0^T \int_{\{x: |d(t,x)| \le \kappa\}} |\tilde{\mc R}_\e (t,x, d_\e) | \cdot |\nabla d(t,x)| \; dxdt \\
=&\; \frac{1}{\e} \int_0^T \int_{-\kappa}^{\kappa} \int_{\{x: d(t,x) = s\}} |\tilde{\mc R}_\e (t,x, s/\e)| \; d\mathcal{H}^{d-1}(dx) ds dt \\
\le&\; \frac{C_\e}{\e} \int_0^T \int_{-\kappa}^{\kappa} e^{-\gamma' |s|/ \e} \mathcal{H}^{d-1}(\{x: d(t,x) = s\}) \; dsdt. 
\end{align*}
Since the surface $\{x: d(t,x) = s\}$ can be represented by 
\[ \{x \in \T^d: d(t,x) = s\} = \{y + n_t(y) s \in \T^d: y \in \Gamma_t\}, \]
the change of variables formula for integrations (cf.\ \cite[Theorem 3.9]{EG}) yields that 
\[ \mathcal{H}^{d-1}(\{x: d(t,x)=s\}) = \int_{\Gamma_t} |\det (\nabla_{\Gamma_t} {\rm Id}(y) + s \nabla_{\Gamma_t} n_t(y))| \; \mathcal{H}^{d-1}(dy), \]
where ${\rm Id}$ is the identity map satisfying ${\rm Id}(y) = y \in \T^d$ for $y \in \Gamma_t$ and $\nabla_{\Gamma_t}$ is the divergence operator on the surface $\Gamma_t$. 
Therefore, letting 
\[ M := \max_{t \in [0,T], y \in \Gamma_t, s \in [-\kappa, \kappa]}  |\det (\nabla_{\Gamma_t} {\rm Id}(y) + s \nabla_{\Gamma_t} n_t(y))| < \infty \]
we have 
\begin{align*} 
|I_1| \le&\; \frac{TMC_\e}{\e} \left(\max_{t \in [0,T]} \mathcal{H}^{d-1}(\Gamma_t)\right) \int_{-\kappa}^{\kappa} e^{-\gamma' |s|/\e} \; ds \\
=&\; TMC_\e \left(\max_{t \in [0,T]} \mathcal{H}^{d-1}(\Gamma_t)\right) \int_{-\kappa/\e}^{\kappa/\e} e^{-\gamma' |\tilde{s}|} \; d\tilde{s}. 
\end{align*}
We here have changed the variable $s/\e$ to $\tilde{s}$. 
We thus obtain $I_1 \to 0$ as $\e \to 0$ by means of the convergence of $C_\e$ to $0$, which implies \eqref{con-decay-tR1} by combining \eqref{decay-tR1}. 

We next prove \eqref{con-decay-tR2}. 
We decompose the integral as 
\begin{align*}
&\; \frac{1}{\e} \int_0^T \int_{\T^d} A (t,x, d_\e) \; dxdt \\
=&\; \frac{1}{\e} \int_0^T \int_{\{x: |d(t,x)| \le \kappa\}} A (t,x, d_\e) \; dxdt + \frac{1}{\e} \int_0^T \int_{\{x: |d(t,x)| > \kappa\}} A (t,x, d_\e) \; dxdt =: I_3 + I_4. 
\end{align*}
We can prove $I_4 \to 0$ as $\e \to 0$ by a similar argument for $I_2$. 
Apply the co-area formula, the change of variables formula, and change the variable $s/\e$ to $\tilde{s}$ to obtain 
\begin{align*}
I_3 =&\; \frac{1}{\e} \int_0^T \int_{-\kappa}^{\kappa} \int_{\{x: d(t,x) = s\}} A(t, x, s/\e) \; \mathcal{H}^{d-1}(dx) ds dt \\
=&\; \frac{1}{\e} \int_0^T \int_{-\kappa}^{\kappa} \int_{\Gamma_t} A(t, y+s n_t(y), s/\e) |\det (\nabla_{\Gamma_t} {\rm Id}(y) + s \nabla_{\Gamma_t} n_t(y))| \; \mathcal{H}^{d-1}(dy) ds dt \\
=&\; \int_0^T \int_{-\kappa/\e}^{\kappa/\e} \int_{\Gamma_t} A(t, y+\e\tilde{s} n_t(y), \tilde{s}) |\det (\nabla_{\Gamma_t} {\rm Id}(y) + \e \tilde{s} \nabla_{\Gamma_t} n_t(y))| \; \mathcal{H}^{d-1}(dy) d\tilde{s} dt. 
\end{align*}
Therefore, we have by the dominated convergence theorem 
\[ I_3 \to \int_0^T \int_\R \int_{\Gamma_t} A(t, y, \tilde{s}) \; \mathcal{H}^{d-1}(dy) d\tilde{s} dt \quad \text{as} \; \; \e \to 0, \]
which yields \eqref{con-decay-tR2} by means of the convergence of $I_4$. 
\end{proof}

\end{document}